\documentclass{article}
\usepackage[T1]{fontenc}
\usepackage[english]{babel}
\usepackage{amsmath, amssymb, amsthm}
\usepackage{mathrsfs}
\usepackage{geometry}
\usepackage{graphicx}
\usepackage{fancyhdr} % intestazioni personalizzate
\usepackage{tikz}
\usetikzlibrary{decorations.markings,arrows.meta}
\usepackage{amscd}
\usepackage{epsf}
\usepackage{booktabs}
\usepackage{array}
\usepackage{multirow}
\usepackage{color}
\usepackage[symbol]{footmisc}
\usepackage{lipsum}
\usepackage{latexsym}
\usepackage{verbatim}
\usepackage{bbm}
\usepackage{setspace}
\usepackage[backref=page]{hyperref}
\renewcommand*{\backref}[1]{}
\renewcommand*{\backrefalt}[4]{%
	\ifcase #1 (Not cited)%
	\or        (Cited on page~#2)%
	\else      (Cited on pages~#2)%
	\fi}
\usepackage{cleveref}
\usepackage{etoolbox}
\usepackage[normalem]{ulem} % per \uline (sottolineatura)
\usepackage{thmtools} % opzionale, vedi sotto
\usepackage{enumitem}
\usepackage{anyfontsize}

\newcommand{\thickhline}{\noalign{\hrule height 1pt}}
\setlist[enumerate,1]{label=\textbf{\arabic*.}, ref=\arabic*, leftmargin=*, itemsep=0.5ex, topsep=0.5ex}

\let\OLDthebibliography\thebibliography
\renewcommand\thebibliography[1]{
  \OLDthebibliography{#1}
  \setlength{\parskip}{2pt}
  \setlength{\itemsep}{0pt plus 0.3ex}
}

\makeatletter
\renewenvironment{proof}[1][\proofname]{%
  \par\pushQED{\qed}%
  \normalfont\topsep6\p@\@plus6\p@\relax
  \trivlist
  \item[\hskip\labelsep
        \textbf{\textup{#1}}\@addpunct{.}]\ignorespaces
}{%
  \popQED\endtrivlist\@endpefalse
}
\makeatother
\newcommand{\mysubjclass}{\textup{2020} \textbf{Mathematics Subject Classification:} 22C05, 53C05, 	53C26, 53C29}
\newcommand{\mykeywords}{\textbf{Keywords:} Hypercomplex structure, Obata connection, Holonomy, Lie groups, HKT, Twisted Calabi-Yau}

\fancypagestyle{firstpagefooter}{
  \fancyhf{} % pulisci header e footer
  \fancyfoot[L]{\scriptsize
    \mysubjclass \\
    \mykeywords \\
    }
}

\hypersetup{
    colorlinks=true,
    linkcolor=blue,    % Link per teoremi e equazioni in blu
    citecolor=red,     % Citazioni in rosso
    urlcolor=blue,     % Link URL in blu
    pdftitle={The holonomy of the Obata connection on Joyce hypercomplex manifolds},
    pdfauthor={Beatrice Brienza, Udhav Fowdar, Giovanni Gentili, Luigi Vezzoni}
}

\newtheorem{thm}{Theorem}[section]
\newtheorem{prop}[thm]{Proposition}
\newtheorem{lem}[thm]{Lemma}
\newtheorem{cor}[thm]{Corollary}

\newtheorem{introthm}{Theorem}

\theoremstyle{definition}
\newtheorem{rmk}[thm]{Remark}
\newtheorem{ex}[thm]{Example}

\title{The holonomy of the Obata connection on Joyce hypercomplex manifolds}

\numberwithin{equation}{section}

\newcommand{\g}{\mathfrak{g}}
\newcommand{\h}{\mathfrak{h}}
\newcommand{\bi}{\mathfrak{b}}
\newcommand{\di}{\mathfrak{d}}
\newcommand{\f}{\mathfrak{f}}

\newcommand{\R}{\mathbb{R}}
\newcommand{\C}{\mathbb{C}}
\newcommand{\SU}{\mathrm{SU}}
\newcommand{\SO}{\mathrm{SO}}
\newcommand{\Sp}{\mathrm{Sp}}
\newcommand{\G}{\mathrm{G}}

\newcommand{\w}{\wedge}
\renewcommand{\H}{\mathbb{H}}

\renewcommand{\epsilon}{\varepsilon}
\renewcommand{\phi}{\varphi}
\renewcommand{\Re}{\mathrm{Re}}

\usetikzlibrary{shapes.misc}
\tikzset{cross/.style={cross out, draw, 
         minimum size=2*(#1-\pgflinewidth), 
         inner sep=0pt, outer sep=0pt},
         cross/.default={3.3pt}}

\def\eqref#1{(\ref{#1})}
\title{\textbf{The holonomy of the Obata connection on Joyce hypercomplex manifolds}}
\author{Beatrice Brienza, Udhav Fowdar, Giovanni Gentili, Luigi Vezzoni}
\date{}

\begin{document}

\maketitle

\begin{abstract}
We study the holonomy of the Obata connection on Joyce hypercomplex manifolds. For all such group manifolds except \( \mathrm{SU}(2n+1) \), we show that the holonomy group is strictly contained in the quaternionic general linear group. The case of \( \mathrm{SU}(2n+1) \) is more subtle: for every \( n > 1 \), we show that there exist infinitely many Joyce hypercomplex structures with Obata holonomy strictly contained in \( \mathrm{GL}(n(n+1),\H) \). On the other hand, Soldatenkov showed that \( \mathrm{SU}(3) \) has Obata holonomy equal to \( \mathrm{GL}(2, \mathbb{H}) \)~\cite{Sol}, and we present here a new example on \( \mathrm{SU}(5)\)  with holonomy equal to \( \mathrm{GL}(6,\H) \). Finally, we investigate Joyce hypercomplex manifolds whose restricted holonomy lies in \( \mathrm{SL}(n, \mathbb{H}) \), yielding new compact examples of twisted Calabi--Yau manifolds.
\end{abstract}

\thispagestyle{firstpagefooter}
\tableofcontents

\section{Introduction}

A \emph{hypercomplex manifold} is a smooth manifold $M^{4n}$ equipped with a triple of complex structures $I, J, K$ satisfying the quaternionic relations
$$
I^2=J^2=K^2=IJK=-\mathrm{Id}\,.
$$
It was shown in \cite{Ob} that every hypercomplex manifold admits a unique torsion-free connection preserving the hypercomplex structure, known as the \emph{Obata connection}.\par
Following the work of Boyer in \cite{Bo}, hypercomplex geometry has attracted considerable attention, especially in relation to \emph{hyperhermitian} structures. In the \emph{hyperk\"ahler} case, the Obata connection coincides with the Levi-Civita connection of the underlying metric, and its holonomy group is contained in the compact symplectic group $\mathrm{Sp}(n)$. However, in the general hypercomplex case, the Obata connection is not compatible with any metric, so  its holonomy group is typically non-compact, lying in the quaternionic general linear group $\mathrm{GL}(n, \mathbb{H})$. \par
Left-invariant hypercomplex structures on Lie groups have appeared in the context of
string theory in \cite{SSTvP} and in a more mathematical framework in the seminal work of Joyce in \cite{Joy}. Joyce showed that if $G$ is a compact semisimple Lie group of rank $r$, then the compact Lie group $\mathbb{T}^{2m - r} \times G$ admits up to an $m^2$-parameter family of non-equivalent left-invariant hypercomplex structures, where $m \in \mathbb{N}$ is the number of $\mathfrak{su}(2)$ summands in the so-called Joyce decomposition of $G$ (see Section~\ref{sec:joyce}). \par
Subsequently, it was shown in \cite{DT,BGP} that any left-invariant hypercomplex structure on a compact Lie group arises via Joyce's construction, and so throughout this paper a compact Lie group endowed with such a structure will be referred to as a \emph{Joyce hypercomplex manifold}. The list of Joyce hypercomplex manifolds $\mathbb{T}^{2m - r} \times G$ with $G$ simple was given in \cite{SSTvP}:
\begin{equation} \label{eqn:classification}
\begin{aligned}
&\SU(2k+1), %\ (k\geq 1),
\ \ S^1 \times \SU(2k), %\ (k\geq 1),
\ \ \mathbb{T}^k \times \SO(2k+1), % \  (k \ge 3), \\
\ \ \mathbb{T}^k \times \Sp(k), % \ (k \ge 2),
\ \ \mathbb{T}^{2k}\times \SO(4k), % \ (k \ge 2),
\\
&
\ \  \mathbb{T}^{2k-1} \times \SO(4k+2), % \ (k\geq 2), \\ 
\ \ \mathbb{T}^2 \times \mathrm{E}_6,
\ \ \mathbb{T}^7 \times \mathrm{E}_7,
\ \ \mathbb{T}^8 \times \mathrm{E}_8,
\ \ \mathbb{T}^4 \times \mathrm{F}_4,
\ \ \mathbb{T}^2 \times \mathrm{G}_2.
\end{aligned}
\end{equation}
An important problem in hypercomplex geometry is determining the holonomy of the associated Obata connection. 
However, very little is known beyond the low-dimensional cases.\par
In quaternionic dimension one, the only Joyce hypercomplex manifold is the Hopf surface $S^1\times {\rm SU}(2)$. In this case, the Obata connection is flat with holonomy group $\mathbb Z$ \cite{SV}.\par
In quaternionic dimension two, the situation is significantly richer. Soldatenkov showed in \cite{Sol} that the holonomy group of the Obata connection on $\SU(3)$ is \textit{equal} to $\mathrm{GL}(2, \mathbb{H})$. The proof first involves showing that the Obata holonomy group acts irreducibly on the tangent space; the argument relies crucially on properties specific to dimension~8. Secondly, one appeals to the classification of irreducible holonomy groups of torsion-free affine connections in \cite{MS}. In the latter classification, only three candidates appear as possible subgroups of $\mathrm{GL}(n, \mathbb{H})$: $\mathrm{Sp}(n)$, $\mathrm{SL}(n, \mathbb{H})$ and $\mathrm{GL}(n, \mathbb{H})$, and it is simply a matter of eliminating the former two.\par 
It was subsequently conjectured (see e.g. \cite{SV}) that for any Joyce hypercomplex manifold in the list \eqref{eqn:classification} (with the exception of the Hopf surface) the holonomy group of the Obata connection is always equal to the full group $\mathrm{GL}(n, \mathbb{H})$. We should point out that even if $M$ is a product manifold in \eqref{eqn:classification}, the Joyce hypercomplex structure is not a product one and as such the Obata holonomy does not necessarily split. \par
Our first main result is the following theorem, which, in particular, disproves the conjectural picture recalled in the paper of Soldatenkov and Verbitsky~\cite{SV}:
\begin{introthm}\label{THM:A}
Let $M$ be a compact Lie group from the list $\eqref{eqn:classification}$, except $\mathrm{SU}(2n+1)$, and let $\nabla$ denote the Obata connection of a left-invariant hypercomplex structure on $M$. Then the Obata holonomy is a proper subgroup of  $\operatorname{GL}(n, \mathbb{H})$, i.e.
\[
\operatorname{Hol}(\nabla) \subsetneqq \operatorname{GL}(n, \mathbb{H}).
\] 
\end{introthm}
Recall that any left-invariant hypercomplex structure on $M$ arises via Joyce's construction. Furthermore, on each manifold $M$, there exist in general infinitely many non-equivalent Joyce hypercomplex structures, and Theorem \ref{THM:A} holds for the Obata connection associated to \emph{each} of these structures.   
The key result to proving the theorem is Lemma~\ref{Lem:technical}, which provides a sufficient condition for the existence of a $\nabla$-parallel subbundle in terms of the Joyce decomposition of $M$. This condition depends on geometric properties of the simple Lie group $G$ only and can be verified directly from its Dynkin diagram. 
Determining the holonomy, however, is still a rather difficult problem. In Section \ref{Subsec:Sp2}, we compute explicitly the holonomy algebra for $\mathbb{T}^2 \times \mathrm{Sp}(2)$. We show that it is an $11$-dimensional subalgebra of the Lie algebra of quaternionic lower triangular matrices (Theorem \ref{thm: sp2T2}). \par

The case of $\mathrm{SU}(2n+1)$ turns out to be rather peculiar and we treat it separately. We show that for $n \geq 2$ there always exist many Joyce hypercomplex structures on $\mathrm{SU}(2n+1)$ for which the Obata holonomy is \textit{strictly} contained in $\mathrm{GL}(n(n+1),\mathbb{H})$ (Theorem \ref{thm:su2n+1}). However, there appear to also exist hypercomplex structures for which the holonomy group is full; we provide such an example on $\mathrm{SU}(5)$ (Theorem \ref{thm: su5}). 
We conjecture that there exist many more such examples, in contrast to those in Theorem \ref{thm:su2n+1}. The above dichotomy does not occur on $\SU(3)$ since the whole $1$-parameter family of Joyce hypercomplex structures have full holonomy \cite{Sol}.

\medskip

It was shown in \cite{BDV,SV} that no Joyce hypercomplex manifold has holonomy group contained in $\mathrm{SL}(n,\mathbb{H})$. However, there are many examples for which the \emph{restricted} holonomy lies in $\mathrm{SL}(n,\mathbb{H})$. For instance, we show that this occurs whenever the abelian summand $\bi$ in the Joyce decomposition of $G$ vanishes (see Section~\ref{sec:joyce}). In that case, $M$ belongs to the list: 
\begin{equation}\label{list}
\mathbb{T}^k\times \mathrm{SO}(2k+1),
\ \  \mathbb{T}^{2k}\times \mathrm{SO}(4k),\ \ \mathbb{T}^k\times \mathrm{Sp}(k), \ \mathbb{T}^7\times \mathrm{E}_7, \ \ \mathbb{T}^8\times \mathrm{E}_8, \ \ \mathbb{T}^4\times \mathrm{F}_4, \ \ \mathbb{T}^2\times \mathrm{G}_2.
\end{equation}

We show that the corresponding Joyce hypercomplex manifold $\mathbb{T}^{2m - r} \times G$ admits left-invariant solutions to the twisted Calabi--Yau system introduced in \cite{GRST}. We recall that a {\em twisted Calabi-Yau structure} on a complex manifold $(M,I)$ is a Hermitian metric $g$ together with a complex volume form $\Psi$ of constant norm such that 
\begin{equation}\label{eqtCY}
\partial \bar\partial \omega=0\,,\quad d\Psi=\theta \wedge \Psi\,,\quad d\theta=0\,,
\end{equation}
where $\omega:=g(I\cdot,\cdot)$ denotes the fundamental $2$-form and $\theta:=Id^{*}\omega$ is the Lee form. Twisted Calabi-Yau manifolds yield solutions to the {\em twisted Hull-Strominger system} \cite{GRST}. 
Known examples include Calabi-Yau manifolds, 
quaternionic Hopf surfaces, and ${\rm SU}(3)$-structures on $\mathbb{T}^3\times S^3$ and on $M\times \R$, where $M$ is a Sasakian $5$-manifold.\par 
Our second main result is the following theorem, which, in particular, provides new examples of twisted Calabi-Yau manifolds:
\begin{introthm}[Theorem \ref{main2}] \label{THM:B}
Let $M$ belong to the list $\eqref{list}$. Then any left-invariant hypercomplex structure $(I,J,K)$ on $M$ has restricted Obata holonomy contained in ${\rm SL}(n,\H)$. Moreover, $(M,I)$ admits a left-invariant solution to the twisted Calabi-Yau system $\eqref{eqtCY}$.  
\end{introthm}

The paper is organised as follows. In Section \ref{Sec:pre} we gather some  preliminary material. We review the \emph{Joyce decomposition} of semisimple Lie algebras and the construction of left-invariant hypercomplex structures. In Section \ref{sec:PrelimLemmas} we prove new technical results on the Lie algebraic structure associated to the Joyce decomposition. This allows us to establish the key Lemma \ref{Lem:technical}. We use this lemma in Section \ref{Sec:holred} to prove Theorem \ref{THM:A}. We also address the remaining case of $\mathrm{SU}(2n+1)$ and provide explicit calculation of the holonomy algebra for the case of $\mathbb{T}^2 \times \mathrm{Sp}(2)$. Finally, Section \ref{Sec:restrhol} is devoted to the proof of Theorem \ref{THM:B}.

\noindent 
\newline 
{\bf Acknowledgments:}
The authors would like to thank A. Fino, G. Grantcharov, S. Salamon, A. Soldatenkov, and M. Verbitsky for their interest in this work and several helpful comments. BB also thanks IMPA and FAMAF for their warm hospitality during the preparation of this work. \\
BB was supported by the Project PRIN 2022 “Real and complex manifolds: Geometry and Holomorphic Dynamics” and by the INdAM - GNSAGA Project CUP E53C24001950001.
UF was supported by the S\~ao Paulo Research Foundation (FAPESP) under grant [2023/12372-1], associated with the BRIDGES Collaboration [2021/04065-6]. BB, GG and LV were partially supported by GNSAGA (INdAM). LV was also supported by the PRIN project \lq\lq Differential-geometric aspects of manifolds via Global Analysis'' 20225J97H5.

\section{Preliminaries}\label{Sec:pre}

In this section, we review the construction and basic properties of Joyce hypercomplex manifolds. We recall the Joyce decomposition of compact semisimple Lie algebras, the resulting left-invariant hypercomplex structures on products $\mathbb{T}^{2m-r}\times G$, and the existence of compatible HKT metrics. We refer the reader to \cite{GP, Joy, OP} for further details. 

\subsection{Joyce decomposition} \label{sec:joyce}
The Joyce construction applies to any compact Lie group $G$. Since, up to a finite cover, any compact Lie group splits as a product of a torus and a compact semisimple Lie group, for our purposes here we may assume without loss of generality that $G$ is semisimple. We set $r=\operatorname{rank}{G}$. \par
Let $H$ be a maximal torus in $G$, and let $\mathfrak{h}$ and $\mathfrak{g}$ denote their respective Lie algebras\footnote{In what follows, we denote the Lie algebra of a Lie group $G$ by the same letter in fraktur font, that is, $\operatorname{Lie}(G) = \mathfrak{g}$.}.
Choose a system of ordered roots $\Delta$ with respect to $\mathfrak{h}_\C$ and fix a maximal positive root $ 
\alpha_1$. Let $\di_1$ be the $\mathfrak{sp}(1)$-subalgebra of $
\mathfrak{g}$ whose complexification is isomorphic to the $\mathfrak{sl}(2,\C)$-subalgebra $[\g_{\alpha_{1}},\g_{-\alpha_{1}}]\oplus\mathfrak{g}_{\alpha_1}\oplus\mathfrak{g}_{-\alpha_1}$, where $
\mathfrak{g}_{\alpha_1}$ and $\mathfrak{g}_{-\alpha_1}$ are the root spaces for $\alpha_1$ and $-\alpha_1$, respectively. Let $\mathfrak{b}_1$ denote the centraliser of $\di_1$. Then there is a real subspace $\mathfrak{f}_1$ of  dimension $4d_1$, for some $d_1\in \mathbb{N}_0$, such that $\mathfrak{g}=\mathfrak{b}_1\oplus\di_1\oplus \mathfrak{f}_1$. More precisely,
\[\f_1=\g \cap \bigoplus_{\substack{\alpha_1 \neq \alpha > 0 \\ \langle \alpha, \alpha_1 \rangle \neq 0}} \mathfrak{g}_{\alpha}\oplus\mathfrak{g}_{-\alpha}.
\]
The subalgebra $\mathfrak{b}_1$ is the direct sum of an abelian Lie algebra and a semisimple Lie algebra $\g'$. If $\g'$ is non trivial, then Joyce re-applies the above decomposition to $\g'$. \par
By a recursive process, one obtains a decomposition of $\g$ of the form \cite[Lemma 4.1]{Joy}:
\begin{equation}\label{eqn:Joycedec}
\mathfrak{g}=\mathfrak{b}\oplus \bigoplus_{i=1}^m\mathfrak{d}_i\oplus \bigoplus_{i=1}^m\mathfrak{f}_i\,,
\end{equation}
where
\begin{enumerate} 
\item $ \mathfrak{b} $ is an abelian subalgebra of dimension $ r-m $, 
\item $ \mathfrak{d}_i\subseteq  \mathfrak{g} $ is a subalgebra isomorphic to $ \mathfrak{su}(2) $ for each $i=1,\dots,m$, 
\item $ \mathfrak{f}_i\subseteq \mathfrak{g} $ are (possibly trivial) subspaces for each $i=1,\dots,m$. 
\end{enumerate} 
Furthermore, the following commutation relations hold:
\begin{enumerate}[label=(J\arabic*),ref=J\arabic*]
\item \label{eqn:joyce1} $[\mathfrak{d}_i,\mathfrak{b}]=0 $, for any $ i=1,\dots,m $;
\item \label{eqn:joyce2} $ [\mathfrak{d}_i,\mathfrak{d}_j]=0 $,  for $ i\neq j $;
\item \label{eqn:joyce3} $ [\mathfrak{d}_i,\mathfrak{f}_j]=0 $,  for $ i<j$;
\item \label{eqn:joyce4} $ [\mathfrak{d}_i,\mathfrak{f}_i]\subseteq \mathfrak{f}_i $, for any $ i=1,\dots,m $; moreover, the action of $\mathfrak{d}_i$ on $\mathfrak{f}_i$ is isomorphic to the direct sum of a finite number of copies of the standard $ \mathfrak{su}(2) $-action on $ \C^2 $.
\end{enumerate}
We will refer to a decomposition as in \eqref{eqn:Joycedec} as a \emph{Joyce decomposition}. Note that Joyce's construction involves the choice of a Cartan subalgebra, a system of positive roots, and a maximal root. While different choices lead to different decompositions, they are nonetheless all isomorphic via an inner automorphism of $\g$. Thus, we may unambiguously talk about \emph{the} Joyce decomposition of a given compact Lie group.\par
For further purposes, we emphasise the link between the Joyce decomposition and the construction of compact quaternion-K\"ahler symmetric spaces, so-called Wolf spaces, as described in \cite{Wo}. This relation has also been observed in \cite{OP, SSTvP}. \par
It is shown in \cite{Wo} that for each compact simple Lie group $G$ there is a unique compact quaternion-K\"ahler symmetric space $M^{4n}=G/H$ obtained as a quotient of $G$ by a subgroup $H=K \cdot  \mathrm{Sp}(1)$:
\begin{equation*}
   M^{4n}:= \frac{G}{K\cdot \mathrm{Sp}(1)},
\end{equation*}
where $\mathrm{Sp}(1)$ is the compact real form of the $\mathrm{SL}(2,\C)$-triple associated with the maximal root of $G$, and $K$ is its centraliser in $G$. This yields the following decomposition of the Lie algebra $\g$:
\begin{equation*}
    \mathfrak{g}=\mathfrak{k}\oplus \mathfrak{su}(2) \oplus \mathfrak{m},
\end{equation*}
where $\mathfrak{m}$ corresponds to the isotropy representation. Since $M$ carries a quaternion-Kähler structure, as an $\mathrm{Sp}(1) \cong \mathrm{SU}(2)$-module we have $\mathfrak{m}\cong \C^2 \oplus \cdots \oplus \C^2$ ($n$ copies) cf. \cite{Sal}. \par
The above isotropy decomposition coincides precisely with the Joyce decomposition at the first level:
set $\mathfrak{d}_1=\mathfrak{su}(2)$ and $\mathfrak{m}=\f_1$. Since $K$ is a compact Lie group (not necessarily simple), one can apply the Joyce decomposition, up to a shift of the index $i$, to $\mathfrak{k}$. Thus, we can summarise the above observation into:
\begin{lem} \label{lem:misha}
The decomposition
\begin{equation}\label{eqn:dec}
    \mathfrak{g} = \Big(\mathfrak{b}\oplus \bigoplus_{i=2}^m\mathfrak{d}_i\oplus \bigoplus_{i=2}^m\mathfrak{f}_i \Big) \oplus \mathfrak{su}(2) \oplus \mathfrak{m}
\end{equation}
coincides with the Joyce decomposition $\eqref{eqn:Joycedec}$ of $\mathfrak{g}$.
\end{lem}
The latter follows from the construction in \cite{Wo} and is well known to the experts.
 \par
\subsection{Joyce hypercomplex structures} \label{sec:hcx}
As above, we denote by $G$ a compact semisimple Lie group of rank $r$, with associated Joyce decomposition \eqref{eqn:Joycedec}.
Let $\mathbb{T}^{2m-r}\cong \mathrm{U}(1)^{2m-r} $ be a $ (2m-r) $-dimensional torus. To simplify the notation, we set $\ell:=2m-r$. Thus, we have an isomorphism:
$$
\ell \mathfrak{u}(1)\oplus \mathfrak{b}\cong \R^m.
$$ 
Note that the choice of such an isomorphism (that is, the choice of a basis of $\ell \mathfrak{u}(1)\oplus \mathfrak{b}$) depends on $m^2$ parameters. Let $ \mathcal{B}=(e_1^1,e_1^2,\dots,e_1^m) $ denote a fixed basis of $ \ell \mathfrak{u}(1)\oplus \mathfrak{b}\cong \R^m $. The Joyce hypercomplex structure on $\mathbb{T}^{\ell }\times G$ constructed from $\mathcal{B}$ is defined on each ``layer'' of the decomposition \eqref{eqn:Joycedec} as follows: 
\begin{itemize}
    \item For every $ i=1,\dots,m $, fix a basis $ (e_2^i,e_3^i,e_4^i)$ of $\mathfrak{d}_i \cong  \mathfrak{su}(2) $ so that $ e_2^i,e_3^i,e_4^i $ satisfy
\begin{align}
\label{eq_basis_su2}
[e_2^i,e_3^i]=2e^i_4\,,&& [e_4^i,e^i_2]=2e^i_3\,,&& [e^i_3,e^i_4]=2e^i_2\,.
\end{align}
In this way, $ (e_1^i,e_2^i,e_3^i,e_4^i)$ is a basis of $ \mathfrak{h}_i:=\R \oplus \mathfrak{d}_i \cong \R \oplus \mathfrak{su}(2)$, which can be regarded as a copy of the space of quaternions. In particular, each $\h_i$ has a natural hypercomplex structure given by:
\begin{align*}
Ie^i_1=e^i_2\,, \ \  Ie^i_3=e^i_4\,, \ \   Je^i_1=e^i_3\,, \ \  Je^i_2=-e^i_4\,,
 \ \  Ke^i_1=e^i_4\,, \ \  Ke^i_2=e^i_3\,.
\end{align*}
\item The action of $I,J,K$ is extended to each $ \mathfrak{f}_i $ by taking advantage of property \eqref{eqn:joyce4}:
\[
If=[e^i_2,f]\,, \quad Jf=[e^i_3,f]\,, \quad Kf=[e^i_4,f]\,,
\]
for each $ f\in \mathfrak{f}_i $.
\end{itemize}
The complex structures $\{I,J,K\}$, defined at the identity, are extended to all of $\mathbb{T}^{\ell}\times G$ by left translation. By appealing to the construction in \cite{Sam},
Joyce shows that these almost complex structures are in fact integrable and define a homogeneous hypercomplex structure on $\mathbb{T}^{\ell}\times G$ \cite{Joy}. We shall call $(\mathbb{T}^{\ell}\times G,I,J,K)$ a \emph{Joyce hypercomplex manifold}. We emphasise that the hypercomplex structure depends on the choice of the basis $\mathcal{B}$ and thus, there is a parameter space of dimension $m^2$ of left-invariant hypercomplex structures on the same underlying manifold $\mathbb{T}^{\ell}\times G$ and many of them are inequivalent.

\begin{rmk}
It is worth pointing out that Joyce also extended the construction of invariant hypercomplex structures to certain homogeneous spaces of the form $\mathbb{T}^s \times (G/H)$ for suitable $s$ and $H$ \cite{Joy}. Although some of our results can be applied to such homogeneous hypercomplex spaces, we shall restrict to the case of trivial isotropy i.e. group manifolds.
\end{rmk}

\subsection{HKT metrics on Joyce hypercomplex manifolds} \label{section:HKT}
Let $B$ denote the negative of the Killing--Cartan form of $\g$. It was shown in \cite{GP} that the Joyce decomposition \eqref{eqn:Joycedec} is $B$-orthogonal. 
Let $(e_2^j,e_3^j,e_4^j)$ denote a standard basis of $\di_j$ satisfying \eqref{eq_basis_su2} such that
\[
B(e_2^j,e_2^j)= B(e_3^j,e_3^j)= B(e_4^j,e_4^j)=\lambda_j^2, \ \ j=1,\dots,m.
\]
In other words, the constants $\lambda_j$ correspond to the (uniform) length of the basis vectors $(e_2^j,e_3^j,e_4^j)$ with respect to  $B$.
Consider now a basis $(e_1^1,\dots,e_1^\ell,e_1^{\ell+1},\dots,e_1^{m})$ of $\ell\mathfrak{u}(1) \oplus \mathfrak{b} \cong \R^m$, where $(e_1^1,\dots,e_1^\ell)$ is a basis of $\ell \mathfrak{u}(1)$ and $(e_1^{\ell +1},\dots,e_1^{m})$ is a $B$-orthogonal basis of $\bi$ such that
\begin{equation}\label{eqn:compatibility}
B(e_1^{\ell+j},e_1^{\ell+j})=\lambda_{\ell+j}^2, \ \ j=1,\dots,m-\ell.
\end{equation}
We then extend $B$ to a positive definite bilinear form $g$ of $\ell \mathfrak{u}(1) \oplus \mathfrak{g}$ by setting:
\[
g(e_1^{j},e_1^{j})=\lambda_{j}^2, \ \ j=1,\dots,\ell.
\]
The latter choice ensures that $(e_1^j,e_2^j,e_3^j,e_4^j)$ have uniform length, and
By construction $g$ defines a bi-invariant metric on $\ell \mathfrak{u}(1) \oplus \mathfrak{g}$ which is hyperhermitian with respect to the Joyce hypercomplex structure $(I,J,K)$ obtained from the basis $(e_1^1,\dots,e_1^{m})$ of $\ell \mathfrak{u}(1) \oplus \mathfrak{b} \cong \R^m$. 
More precisely, the metric $g$ satisfies:
\begin{enumerate}
    \item $g_{|\g}=B$, 
    \item the decomposition $\ell \mathfrak{u}(1) \oplus \mathfrak{b} \oplus \bigoplus_{i=1}^m \mathfrak{d}_i \oplus \bigoplus_{i=1}^m \mathfrak{f}_i$ is $g$-orthogonal,
    \item $g$ is bi-invariant,
    \item the hyperhermitian structure $(g,I,J,K)$ is strong HKT \footnote{A hyperhermitian structure $(I,J,K,g)$ is said to be hyperK\"ahler with torsion (HKT) \cite{HP} if $d^c_I\omega_I=d^c_J\omega_J=d^c_K\omega_K$, where $\omega_L=g(L \cdot,\cdot)$ and $d^c_L=L^{-1}dL$. An HKT structure is strong HKT if $dd^c_I\omega_I=0$.}.
\end{enumerate}
This has been observed in \cite{GP,OP}. \par

We emphasise that the requirement \eqref{eqn:compatibility} imposes some constraints on the possible choices of $(e_1^{\ell +1},\dots,e_1^m)$; indeed not all the Joyce hypercomplex structures are compatible with an extension of the Killing--Cartan form. However, such a choice can always be made, i.e. for each compact simple Lie group there exists at least one Joyce hypercomplex structure compatible with the extension of the Killing--Cartan form as above. We also note that in general the choices of the bases $(e_1^1,\dots,e_1^\ell)$ and $(e_1^{\ell +1}, \dots, e_1^{m})$ are not unique.
When $\mathfrak{b} = 0$, the condition \eqref{eqn:compatibility} is vacuous, and as such any Joyce hypercomplex structure is compatible with an extension of the Killing--Cartan form $B$ as above.

\begin{rmk} \label{rmk:uniqueness}
Note that when \(\mathfrak{b} = 0\), there is exactly one left-invariant hypercomplex structure on the universal cover \(\mathbb{R}^m \times \widetilde{G}\), up to isomorphisms.
Hence, given any two Joyce hypercomplex structures \((I,J,K)\) and \((\hat{I}, \hat{J}, \hat{K})\) with Obata connections \(\nabla\) and \(\hat{\nabla}\), respectively, one has
\[
\mathfrak{hol}(\nabla) \cong \mathfrak{hol}(\hat{\nabla}).
\]
Here we are using the fact that the restricted holonomy group coincides with the holonomy group of the induced Obata connection on the universal cover \(\mathbb{R}^m \times \widetilde{G}\).
\end{rmk}

\section{Technical Lemmas} \label{sec:PrelimLemmas}
In this section, we prove several technical lemmas, based on explicit computations, which  will be crucial in the subsequent sections. We follow the same notation as in the previous section.

\begin{lem}\label{Lem:structure}
Let $G$ be a compact semisimple Lie group and let $\g$ be its Lie algebra. 
With respect to the decomposition $\eqref{eqn:Joycedec}$ we have
\begin{enumerate}
\item $[\mathfrak{b},\mathfrak{f}_j]\subseteq \mathfrak{f}_j$ for all $j=1,\dots,m$;
\item $[\mathfrak{d}_i,\mathfrak{f}_j]\subseteq \mathfrak{f}_j$ for $i>j$;
\item $[\mathfrak{f}_i,\mathfrak{f}_j]\subseteq \mathfrak{f}_i$ for $i< j$;
\item $[\mathfrak{f}_i,\mathfrak{f}_i]\subseteq \mathfrak{b}\oplus  \bigoplus_{k\geq i}  \mathfrak{d}_k\oplus  \bigoplus_{k\geq i}   \mathfrak{f}_k$.
\end{enumerate}
\end{lem}
\begin{proof}
Let $(I,J,K)$ be a Joyce hypercomplex structure on the Lie algebra $\ell \mathfrak{u}(1)\oplus \mathfrak{g}$ (see Section \ref{sec:hcx}). \par
Set $h:=g_E \oplus B$, where $g_E$ is the Euclidean metric on $\ell \mathfrak{u}(1)$ and $B$ is the negative of the Killing--Cartan form of $\g$. Then $h$ is a bi-invariant metric on $\ell \mathfrak{u}(1)\oplus \mathfrak{g}$ such that the decomposition
\begin{equation} \label{eqn:dec2}
\ell \mathfrak{u}(1)\oplus \mathfrak{b} \oplus \bigoplus_{i=1}^m \mathfrak{d}_i \oplus \bigoplus_{i=1}^m \mathfrak{f}_i,
\end{equation}
is $h$-orthogonal \cite{GP}. Observe that we are not assuming $h$ to be compatible with the hypercomplex structure $(I, J, K)$. \par
The proof of the lemma is a direct computation, using the bi-invariant metric $h$ on $\ell \mathfrak{u}(1) \oplus \mathfrak{g}$, together with repeated application of the definition of the hypercomplex structure $(I, J, K)$ (see Section~\ref{sec:hcx}) and the properties~\eqref{eqn:joyce1}--\eqref{eqn:joyce4}. \par

We now prove statement \textbf{1}. Fix $b \in \mathfrak{b}$ and $f_j \in \mathfrak{f}_j$. Then
\[
h([b, f_j], X) = -h(f_j, [b, X]) = 0
\]
for all $X \in \ell \mathfrak{u}(1) \oplus \mathfrak{b} \oplus \bigoplus_{k=1}^m \mathfrak{d}_k$ using \eqref{eqn:joyce1}. On the other hand, if $X \in \mathfrak{f}_k$ with $k < j$, we compute:
\[
\begin{split}
h([b, f_j], X) &= -h([b, f_j], I^2 X) \\
&= -h([b, f_j], [e_2^k, I X]) \\
&= -h([[b, f_j], e_2^k], I X) \\
&= -h([[b, e_2^k], f_j], I X) - h([b, [f_j, e_2^k]], I X) = 0,
\end{split}
\]
where we used~\eqref{eqn:joyce1} and~\eqref{eqn:joyce3}. Similarly, if $X \in \mathfrak{f}_k$ with $k > j$, we have:
\[
\begin{split}
h([b, f_j], X) &= -h([b, I^2 f_j], X) \\
&= -h([b, [e_2^j, I f_j]], X) \\
&= -h([[b, e_2^j], I f_j], X) - h([e_2^j, [b, I f_j]], X) \\
&= -h([e_2^j, [b, I f_j]], X) \\
&= \ \ \,  h([b, I f_j], [e_2^j, X]) = 0,
\end{split}
\]
again by~\eqref{eqn:joyce1} and~\eqref{eqn:joyce3}. This concludes the proof of statement \textbf{1}.

We now prove statement \textbf{3}. Fix $f_i \in \mathfrak{f}_i$ and $f_j \in \mathfrak{f}_j$. Let $X \in \mathfrak{g}$. Then:
\begin{equation} \label{eq:strJoyce}
\begin{split}
h([f_i, f_j], X) &= -h([I^2 f_i, f_j], X) \\
&= -h([[e_2^i, I f_i], f_j], X) \\
&= -h([[e_2^i, f_j], I f_i], X) - h([e_2^i, [I f_i, f_j]], X) \\
&= \ \ \, h([[I f_i, f_j], e_2^i], X) \\
&= \ \ \, h([I f_i, f_j], [e_2^i, X]),
\end{split}
\end{equation}
where we used~\eqref{eqn:joyce3} in the fourth line. It follows from~\eqref{eq:strJoyce} that $h([f_i, f_j], X) = 0$ whenever $X \in \ell \mathfrak{u}(1) \oplus \mathfrak{b} \oplus \bigoplus_{k \neq i} \mathfrak{d}_k \oplus \bigoplus_{l > i} \mathfrak{f}_l$. If instead $X \in \mathfrak{d}_i$, then
\[
h([f_i, f_j], X) = h(f_i, [f_j, X]) = 0,
\]
again by~\eqref{eqn:joyce3}. Now suppose $X \in \mathfrak{f}_l$ with $l < i$. Then:
\[
\begin{split}
h([f_i, f_j], X) &= -h([f_i, f_j], I^2 X) \\
&= -h([f_i, f_j], [e_2^l, I X]) \\
&= -h([[f_i, f_j], e_2^l], I X) \\
&= -h([[f_i, e_2^l], f_j], I X) - h([f_i, [f_j, e_2^l]], I X) = 0,
\end{split}
\]
where the final equality follows from~\eqref{eqn:joyce3} and the fact that $l < i < j$. Hence, $[f_i, f_j]$ can only have non-zero components in $\mathfrak{f}_i$, completing the proof of statement \textbf{3}.

We now prove statement \textbf{2} using statement \textbf{3}. Let $d_i \in \mathfrak{d}_i$ and $f_j \in \mathfrak{f}_j$. Then, for all $X \in \ell \mathfrak{u}(1) \oplus \mathfrak{b} \oplus \bigoplus_{k=1}^m \mathfrak{d}_k$, we have:
\[
h([d_i, f_j], X) = -h(f_j, [d_i, X]) = 0,
\]
where we used~\eqref{eqn:joyce1},~\eqref{eqn:joyce2}, the inclusion $[\mathfrak{d}_i, \mathfrak{d}_i] \subset \mathfrak{d}_i$, and the $h$-orthogonality of the decomposition~\eqref{eqn:dec2}. On the other hand, by statement \textbf{3}, we also get:
\[
h([d_i, f_j], X) = h(d_i, [f_j, X]) = 0
\]
for all $X \in \mathfrak{f}_l$ with $l \neq j$, where we used again the $h$-orthogonality of the decomposition~\eqref{eqn:dec2}. This proves statement \textbf{2}.

Finally, to prove statement \textbf{4}, let $f_i, f_i' \in \mathfrak{f}_i$. Then:
\[
h([f_i, f_i'], X) = h(f_i, [f_i', X]) = 0
\]
if $X \in \ell \mathfrak{u}(1)$ or $X \in \mathfrak{d}_k$ with $k < i$. On the other hand, if $X \in \mathfrak{f}_l$ with $l < i$, then by statement \textbf{3}, $[f_i', X] \in \mathfrak{f}_l$, and hence:
\[
h([f_i, f_i'], X) = h(f_i, [f_i', X]) = 0,
\]
where the last equality follows by the $h$-orthogonality of the decomposition~\eqref{eqn:dec2}.
\end{proof}

\begin{lem}\label{Lem:e_1hyp}
Let $(\mathbb{T}^{\ell }\times G, I,J,K)$ be a Joyce hypercomplex manifold. Then every vector $b\in \ell \mathfrak{u}(1)\oplus\mathfrak{b}$ is hyper-holomorphic.
\end{lem}
\begin{proof}
The claim is obvious for $b\in \ell \mathfrak{u}(1)$ since it lies in the center of the Lie algebra. Thus, we only need to prove it for $b\in \mathfrak{b}$. It suffices to show that $b$ is $I$-holomorphic as the proof is analogous for $J$ and $K$. We need to check that
\[
0=(\mathcal{L}_{b}I)X=[b,IX]-I[b,X]
\]
for all $X\in \mathfrak{g}$. The assertion is clear if $X\in \ell  \mathfrak{u}(1) \oplus \mathfrak{b} \oplus \bigoplus_{j=1}^m \mathfrak{d}_j$, so let us assume that $X\in \mathfrak{f}_k$. By assertion \textbf{1} of Lemma \ref{Lem:structure} we have $[b,X]\in \mathfrak{f}_k$. Therefore,
\[
[b,IX]-I[b,X]=[b,[e_2^k,X]]-[e_2^k,[b,X]]=[[e_2^k,b],X]=0\,,
\]
concluding the proof.
\end{proof}

\begin{rmk}
Lemma~\ref{Lem:e_1hyp} generalises  \cite[Proposition~4.2~(1)]{Sol}, where the result is proved for the Euler vector field $\mathcal{E}=-e_1^1$ on \(\mathrm{SU}(3)\) . 
\end{rmk}

Using Lemma \ref{Lem:structure}, we explicitly compute the covariant derivative of $e_1^j$ with respect to the Obata connection for  $j=1,\dots,m$, which will be useful later.

\begin{lem}\label{lem:nablae1}
Let $(\mathbb{T}^{\ell } \times G,I,J,K)$ be a Joyce hypercomplex manifold. Then
\[
\nabla_X e_1^i=
\begin{cases}
-X & \text{if }X \in \mathfrak{h}_i\oplus \mathfrak{f}_i\,,\\
0 & \text{otherwise}\,,
\end{cases}
\]
where $\h_i=\langle e^i_1 \rangle_\R \oplus \di_i$.
\end{lem}
\begin{proof}
We recall the formula for the Obata connection derived in \cite[(2.5)]{Sol}:
\begin{equation}
\nabla_X Y = \frac{1}{2}([X,Y]+I[IX,Y]-J[X,JY]+K[IX,JY]). \label{equ: obata connection sol}
\end{equation}
Using the latter we have
\[
\begin{split}
\nabla_X e_1^i&= \frac{1}{2}([X,e_1^i]+I[IX,e_1^i]-J[X,Je_1^i]+K[IX,Je_1^i])\\
&= \frac{1}{2}(-J[X,e_3^i]+K[IX,e_3^i])\,,
\end{split}
\]
where the last equality follows from Lemma \ref{Lem:e_1hyp}. \par
If $X\in \mathfrak{h}_k$, for $k\neq i$, then by \eqref{eqn:joyce1} and \eqref{eqn:joyce2} we clearly have $\nabla_X e_1^i=0$. If $X\in \mathfrak{h}_i$ then one can easily verify that
\[
\nabla_X e_1^i= \frac{1}{2}(-J[X,e_3^i]+K[IX,e_3^i])=-X\,.
\]
Now, assume $X\in \mathfrak{f}_k$. If $k>i$, then $\nabla_Xe_1^i=0$ by \eqref{eqn:joyce3}. If $k<i$ using \eqref{eqn:joyce2} we get
\[
K[IX,e_3^i]=K[[e_2^k,X],e_3^i]=K[[e_2^k,e_3^i],X]+K[e_2^k,[X,e_3^i]]=K[e_2^k,[X,e_3^i]]\,.
\]
From statement \textbf{2} of Lemma \ref{Lem:structure}, $[X,e_3^i]\in \mathfrak{f}_k$ and so
\[
\nabla_X e_1^i= \frac{1}{2}(-J[X,e_3^i]+K[e_2^k,[X,e_3^i]])= \frac{1}{2}(-J[X,e_3^i]+KI[X,e_3^i]])=0\,.
\]
Finally, if $k=i$ then
\[
\nabla_X e_1^i= \frac{1}{2}(J^2X-KJIX)=-X\,,
\]
as we wanted to show.
\end{proof}
\begin{cor}
On every Joyce hypercomplex manifold there exists a unique vector field $\mathcal{E}$ such that $\nabla \mathcal{E}=\mathrm{Id}$. In particular, the Obata connection can only preserve tensors of type $(k,k)$.
\end{cor}
\begin{proof}
Let $\mathcal{E}:=-\sum_{j=1}^m e_1^j$, then clearly $\nabla \mathcal{E}=\mathrm{Id}$ (see Lemma \ref{lem:nablae1}). The uniqueness of $\mathcal{E}$, as well as the fact that the Obata connection can preserve only tensors of type $(k,k)$, has already been observed in \cite[Remark 4.3]{Sol}.
\end{proof}

\begin{lem}\label{Lem:technical}
Let $M$ be a Joyce hypercomplex manifold of quaternionic dimension $n>1$. If there exists an index $i=1, \dots, m$ such that $\mathfrak{f}_i=0$ in the Joyce decomposition of the semisimple factor of $M$, then the holonomy of the Obata connection is strictly contained in $\mathrm{GL}(n,\mathbb{H})$.
\end{lem}
\begin{proof}
As an immediate consequence of Lemma \ref{lem:nablae1}, we obtain that if there exists an index $i = 1, \dots, m$ such that $\mathfrak{f}_i = 0$, then the subspace $\mathfrak{h}_i = \langle e_1^i \rangle_{\mathbb{H}}$ is preserved by the Obata connection. \par
In particular, since $n > 1$, $\mathrm{Hol}(\nabla) \subsetneqq \mathrm{GL}(n,\mathbb{H})$. Indeed, let $H_i$ denote the corresponding left-invariant and parallel\footnote{A subbundle $H \subset \mathrm{T}M$ is called \emph{parallel} if it is preserved by the connection, i.e., $\nabla_X Y \in \Gamma(H)$ for all vector fields $X \in \Gamma(TM)$ and $Y \in \Gamma(H)$.} subbundle. Since $H_i$ is $I$, $J$, $K$ invariant and parallel, it is preserved under parallel transport. It follows that the holonomy group cannot act transitively on $\mathbb{H}^n \setminus \{0\}$, and hence cannot coincide with $\mathrm{GL}(n,\mathbb{H})$. 
\end{proof}

\section{Holonomy reduction}\label{Sec:holred}
In this section, we prove Theorem \ref{THM:A}. We compute explicitly the holonomy algebra of the Obata connection on $\mathbb T^2\times {\rm Sp}(2)$, providing a concrete illustration of reducible holonomy. 
Finally, we analyse the case of ${\rm SU}(2n+1)$.

\subsection{Proof of Theorem \ref{THM:A}}
Recall that every left-invariant hypercomplex structure on a compact Lie group corresponds to a Joyce one. Moreover, from \cite{SV} we also know that the Hopf surface is Obata flat. Theorem \ref{THM:A} is an immediate consequence of these facts together with Theorem \ref{main} below.

\begin{thm} \label{main}
Let $M$ be a compact Lie group from the list $\eqref{eqn:classification}$, except from $S^1 \times \mathrm{SU}(2)$ and $\mathrm{SU}(2n+1)$, and let $\nabla$ denote the Obata connection of a Joyce hypercomplex structure on $M$. Then there exists a left-invariant subbundle of $TM$ preserved by $\nabla$. In particular,
\[
\operatorname{Hol}(\nabla) \subsetneqq \operatorname{GL}(n, \mathbb{H}).
\] 
\end{thm}
\begin{proof}
Let $M = \mathbb{T}^{\ell} \times G$ be as in the hypothesis of the theorem, and let $(I,J,K)$ be a Joyce hypercomplex structure on $M$. 
From the results in Section \ref{Sec:pre}, the Lie algebra of $M$ admits the following decomposition:
\[
\ell  \mathfrak{u}(1) \oplus \mathfrak{g} = \ell \mathfrak{u}(1) \oplus \mathfrak{b} \oplus \bigoplus_{i=1}^m \mathfrak{d}_i \oplus \bigoplus_{i=1}^m \mathfrak{f}_i=\bigoplus_{i=1}^m \mathfrak{h}_i \oplus \bigoplus_{i=1}^m \mathfrak{f}_i,
\]
where we set
\[
\mathfrak{h}_i := \mathbb{R} \oplus \mathfrak{d}_i = \langle e_1^i, e_2^i, e_3^i, e_4^i \rangle_{\mathbb{R}}.
\]
Note that since the Hopf surface has been excluded, any manifold $M$ considered here has quaternionic dimension strictly greater than $1$.

\smallskip
We apply Lemma \ref{Lem:technical} to prove that, for any $M$ in the theorem, there exists a left-invariant subbundle preserved by the Obata connection. Note that the condition established in Lemma \ref{Lem:technical} depends only on the Joyce decomposition of the Lie group $G$ and not on the choice of the Joyce hypercomplex structure. Therefore, it suffices to check that the Joyce decomposition of the Lie algebras of the following compact simple Lie groups contains a non trivial $\mathfrak{d}_i$ with corresponding trivial $\mathfrak{f}_i$:
\[
\begin{split}
&\SU(2k) \  (k \ge 2), \ \  \SO(2k+1) \  (k \ge 3), \ \ \SO(4k) \  (k\ge 2),\\
&\SO(4k+2) \  (k\ge 2), \ \  \Sp(k) \  (k \ge 2), \\
&\mathrm{E}_6, \ \ \mathrm{E}_7, \ \  \mathrm{E}_8, \ \ \mathrm{F}_4, \ \  \G_2.
\end{split}
\]
The above restrictions on $k$ are imposed to avoid Lie algebra isomorphisms and to exclude the case of $\mathrm{SU}(2)$. To prove this, we describe below two iterative reduction procedures, one using the Wolf spaces, and the other using Dynkin diagrams following \cite{OP}.
\begin{itemize}
    \item $\mathrm{G}_2$: $\frac{\mathrm{G}_2}{\mathrm{Sp}(1)\cdot \mathrm{Sp}(1)}$ is a Wolf space of quaternionic dimension $2$. By Lemma \ref{lem:misha} the Lie algebra $\mathfrak{g}_2$ has the following decomposition \[\mathfrak{g}_2=\mathfrak{d}_1 \oplus \mathfrak{d}_2 \oplus \mathfrak{f}_1.\] Therefore, $\di_2 \neq 0$ but $\f_2=0$.
    \item $\Sp(k)$: We prove the result by induction on $k$. For $k=2$, the space $\frac{\mathrm{Sp}(2)}{\mathrm{Sp}(1) \cdot \mathrm{Sp}(1)}$ is a Wolf space of quaternionic dimension $1$. By Lemma \ref{lem:misha}, the Lie algebra $\mathfrak{sp}(2)$ admits the decomposition
\[
\mathfrak{sp}(2) = \mathfrak{d}_1 \oplus \mathfrak{d}_2 \oplus \mathfrak{f}_1.
\]
Therefore, $\mathfrak{d}_2 \neq 0$ but $\mathfrak{f}_2 = 0$.
By induction, assume that the Joyce decomposition of $\mathrm{Sp}(k)$ has a trivial summand $\mathfrak{f}_i$ for some $i$. Since the quotient $
\frac{\mathrm{Sp}(k+1)}{\mathrm{Sp}(k) \cdot \mathrm{Sp}(1)}
$
is a Wolf space of quaternionic dimension $k$, we have
\[
\mathfrak{sp}(k+1) = \mathfrak{sp}(k) \oplus \mathfrak{d}_1 \oplus \mathfrak{f}_1.
\]
By the inductive hypothesis and Lemma \ref{lem:misha}, it follows that there exists an index $i$ such that $\mathfrak{f}_i$ is trivial in the Joyce decomposition of $\mathfrak{sp}(k)$, which implies that $\mathfrak{f}_{i+1}$ is trivial in the Joyce decomposition of $\mathfrak{sp}(k+1)$.
\item $\mathrm{F}_4$: $\frac{\mathrm{F}_4}{\mathrm{Sp}(3)\cdot \mathrm{Sp}(1)}$ is a Wolf space of quaternionic dimension $7$ and, accordingly, $$\mathrm{Lie}(F_4)=\mathfrak{sp}(3) \oplus\mathfrak{d}_1 \oplus \mathfrak{f}_1.$$ Since the Joyce decomposition of  $\mathfrak{sp}(3)$ has a trivial $\mathfrak{f}_i$ summand, $\mathfrak{f}_{i+1}$ is trivial in the Joyce decomposition of $\mathrm{Lie}(F_4)$, again by Lemma \ref{lem:misha} \footnote{We write $\mathrm{Lie}(F_4)$ instead of $\mathfrak{f}_4$ to avoid ambiguity with the notation used in the Joyce decomposition.
}.
\item $\SO(2k+1)$: the Joyce decomposition of $\mathfrak{so}(2k+1)$ can most easily be described using extended Dynkin diagrams \cite{OP}. The extended Dynkin diagram of ${\rm B}_k$, associated to $\mathfrak{so}(2k+1)$, is given by
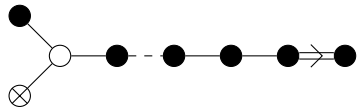
\begin{center}
\begin{tikzpicture}[scale=0.75, every node/.style={circle, draw, minimum size=8pt, inner sep=0pt}]
  % Nodi principali
  \node[fill=black] (a) at (0.2929,0.7071) {};
  \node[fill=white] (b) at (1,0) {};
  \node[fill=white] (c) at (0.2929,-0.7071) {};
  \draw (0.2929,-0.7071) node[cross] {};

  % Nodi orizzontali
  \node[fill=black] (d) at (2,0) {};
  \node[fill=black] (e) at (3,0) {};
  \node[fill=black] (f) at (4,0) {};
  \node[fill=black] (g) at (5,0) {};

  % Nodo finale con doppia freccia
  \node[fill=black] (h) at (6,0) {};

  % Connessioni verticali
  \draw (a) -- (b);
  \draw (c) -- (b);

  % Connessioni orizzontali
  \draw (b) -- (d);
  \draw[dashed] (d) -- (e);
  \draw (e) -- (f);
  \draw (f) -- (g);

  % Doppia freccia -><-
  \draw[double distance=2pt] (g) -- (h);
  \draw (5.4,0.2) -- (5.6,0) -- (5.4,-0.2);
\end{tikzpicture}
\end{center}
where the coloured and uncoloured dots correspond to the simple roots of ${\rm B}_k$ and the additional crossed dot corresponds to the maximal root. \par
At the first level of the Joyce decomposition we obtain $\mathfrak{g}=\mathfrak{b}_1\oplus \mathfrak{d}_1 \oplus \mathfrak{f}_1$, where $\mathfrak{b}_1$ is the centraliser of the $\mathfrak{su}(2)=\mathfrak{d}_1$ generated by the maximal root (see Section \ref{sec:joyce}). The coloured sub-diagram corresponds precisely to the Dynkin diagram of $\mathfrak{b}_{1}$ \cite{OP} (this can also be seen from the classification of Wolf spaces). 
In this case, the sub-diagram is disconnected: one component corresponds to $\mathfrak{so}(2k-3)$ and the isolated black vertex to a $\mathfrak{d}_2$ with associated trivial $\mathfrak{f}_2$. 
\item $\SO(4k), \ \SO(4k+2)$: these cases can be treated very similarly to the previous one, and we shall treat them together. The extended Dynkin diagram associated to ${\rm D}_{2k}$ (${\rm D}_{2k+1}$), corresponding to $\mathfrak{so}(4k)$ \big($\mathfrak{so}(4k+2)$\big), is given by
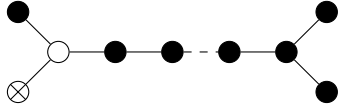
\begin{center}
\begin{tikzpicture}[scale=0.75, every node/.style={circle, draw, minimum size=8pt, inner sep=0pt}]
  % Nodo centrale
  \node[fill=white] (joint) at (0,0) {};
  \node[fill=black] (topLeft) at (-0.7071,0.7071) {};
  \node[fill=white] (bottomLeft) at (-0.7071,-0.7071) {};
  \draw (-0.7071,-0.7071) node[cross] {};

  % Catena orizzontale
  \node[fill=black] (a) at (1,0) {};
  \node[fill=black] (b) at (2,0) {};
  \node[fill=black] (c) at (3,0) {};
  \node[fill=black] (d) at (4,0) {};

  % Estensioni a destra
  \node[fill=black] (topRight) at ({4 + 0.7071}, 0.7071) {};
  \node[fill=black] (bottomRight) at ({4 + 0.7071}, -0.7071) {};

  % Connessioni
  \draw (joint) -- (topLeft);
  \draw (joint) -- (bottomLeft);
  \draw (joint) -- (a);
  \draw (a) -- (b);
  \draw[dashed] (b) -- (c);
  \draw (c) -- (d);
  \draw (d) -- (topRight);
  \draw (d) -- (bottomRight);
\end{tikzpicture}
\end{center}
At the first level of the Joyce decomposition we have that the centraliser $\mathfrak{b}_1$ of the $\mathfrak{su}(2)$-copy corresponding to the maximal root is disconnected. One component corresponds to $\mathfrak{so}(4k-4)$ \big($\mathfrak{so}(4k-2)$\big) and the isolated black vertex to a $\mathfrak{d}_2$ with trivial $\mathfrak{f}_2$.
\item $\mathrm{E}_7$: the extended Dynkin diagram of $\mathrm{E}_7$ is given by 
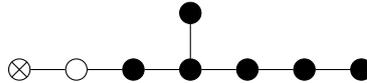
\begin{center}
\begin{tikzpicture}[scale=0.75, every node/.style={circle, draw, minimum size=8pt, inner sep=0pt}]

% === E_7^{(1)} (diagramma a sinistra) ===
  % Nodi orizzontali
  \node[fill=white] (a1) at (0,0) {};
  \draw (0,0) node[cross] {};
  \node[fill=white] (a2) at (1,0) {};
  \node[fill=black] (a3) at (2,0) {};
  \node[fill=black] (a4) at (3,0) {};
  \node[fill=black] (a5) at (4,0) {};
  \node[fill=black] (a6) at (5,0) {};
  \node[fill=black] (a7) at (6,0) {};

  % Nodo verticale
  \node[fill=black] (top1) at (3,1) {};

  % Connessioni orizzontali
  \draw (a1) -- (a2);
  \draw (a2) -- (a3);
  \draw (a3) -- (a4);
  \draw (a4) -- (a5);
  \draw (a5) -- (a6);
  \draw (a6) -- (a7);

  % Connessione verticale
  \draw (a4) -- (top1);

  % Etichetta sotto
  \end{tikzpicture}
\end{center}
The colored sub-diagram corresponds to $\mathrm{D}_6$, and, therefore, the centraliser of the $\mathfrak{su}(2)$-copy corresponding to the maximal root is isomorphic to $\mathfrak{so}(12)$. As seen above the Joyce decomposition of $\mathfrak{so}(12)$ contains a summand $\mathfrak{d}_i$ with  trivial $\mathfrak{f}_i$; the same holds for $\mathrm{E}_7$. 
\item $\mathrm{E}_8$: this case can be treated similarly to the previous one. The extended Dynkin diagram of $\mathrm{E}_8$ is given by
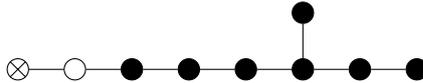
\begin{center}
\begin{tikzpicture}[scale=0.75, every node/.style={circle, draw, minimum size=8pt, inner sep=0pt}]
  % === E_8^{(1)} ===
  \node[fill=white] (b1) at (0,0) {}; \draw (0,0) node[cross] {};
  \node[fill=white] (b2) at (1,0) {};
  \node[fill=black] (b3) at (2,0) {};
  \node[fill=black] (b4) at (3,0) {};
  \node[fill=black] (b5) at (4,0) {};
  \node[fill=black] (b6) at (5,0) {};
  \node[fill=black] (b7) at (6,0) {};
  \node[fill=black] (b8) at (7,0) {};
  \node[fill=black] (top2) at (5,1) {};

  \draw (b1) -- (b2);
  \draw (b2) -- (b3);
  \draw (b3) -- (b4);
  \draw (b4) -- (b5);
  \draw (b5) -- (b6);
  \draw (b6) -- (b7);
  \draw (b7) -- (b8);
  \draw (b6) -- (top2);
\end{tikzpicture}
\end{center}

Since the colored Dynkin diagram of $\mathrm{E}_8$ corresponds to that of $\mathrm{E}_7$, then $\mathrm{E}_7$ is the centraliser of the $\mathfrak{su}(2)$-copy corresponding to the maximal root. Since the Joyce decomposition of $\mathrm{E}_7$ contains a summand $\mathfrak{d}_i$ with trivial $\mathfrak{f}_i$, it follows  that the same holds for $\mathrm{E}_8$.
\item $\SU(2k)$: For $\SU(2k)$, the procedure is slightly different. In each of the previous cases, the subalgebra $\mathfrak{b}_1$ was either simple or of the form $\mathfrak{su}(2) \oplus \mathfrak{g}'$, with $\mathfrak{g}'$ simple. In the present case, considering the extended Dynkin diagram of $\mathrm{A}_{2k-1}$, associated to $\mathfrak{sl}(2k)$:
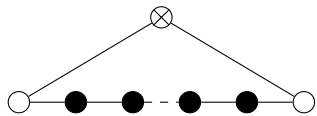
\begin{center}
\begin{tikzpicture}[scale=0.75, every node/.style={circle, draw, minimum size=8pt, inner sep=0pt}]

  % Nodi sulla base
  \node[fill=white] (a) at (0,0) {};
  \node[fill=black] (b) at (1,0) {};
  \node[fill=black] (c) at (2,0) {};
  \node[fill=black] (d) at (3,0) {};
  \node[fill=black] (e) at (4,0) {};
  \node[fill=white] (f) at (5,0) {};

  % Nodo in alto con croce
  \node[fill=white] (top) at (2.5,1.5) {};
  \draw (2.5,1.5) node[cross] {};

  % Connessioni orizzontali
  \draw (a) -- (b);
  \draw (b) -- (c);
  \draw[dashed] (c) -- (d);
  \draw (d) -- (e);
  \draw (e) -- (f);

  % Connessioni verso l'alto
  \draw (a) -- (top);
  \draw (f) -- (top);
\end{tikzpicture}
\end{center}
we have to remove two vertices, resulting in the Dynkin diagram of $\mathrm{A}_{2k-3}$. The reason is that the centraliser of the maximal root in $\mathfrak{su}(2k)$ is given by $\mathfrak{u}(1) \oplus \mathfrak{su}(2k-2)$. Therefore, two simple roots must be removed rather than just one unlike in the previous cases. 
Repeating this process with $\mathfrak{su}(2k-2)$, we eventually arrive to an $\mathfrak{su}(2)$ with  trivial $\mathfrak{f}_i$.
\item ${\rm E}_6$: for ${\rm E}_6$, the associated Dynkin diagram is given by 
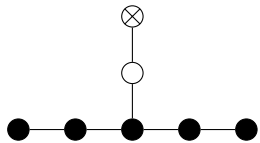
\begin{center}
\begin{tikzpicture}[scale=0.75, every node/.style={circle, draw, minimum size=8pt, inner sep=0pt}]

  % Nodi orizzontali
  \node[fill=black] (a) at (0,0) {};
  \node[fill=black] (b) at (1,0) {};
  \node[fill=black] (c) at (2,0) {};
  \node[fill=black] (d) at (3,0) {};
  \node[fill=black] (e) at (4,0) {};

  % Nodo verticale sopra al nodo centrale
  \node[fill=white] (f) at (2,1) {};
  \node[fill=white] (g) at (2,2) {};
  \draw (2,2) node[cross] {};

  % Connessioni orizzontali
  \draw (a) -- (b) -- (c) -- (d) -- (e);

  % Connessioni verticali
  \draw (c) -- (f) -- (g);
\end{tikzpicture}
\end{center}
and the colored sub-diagram corresponds to ${\rm A}_5$. Therefore, the centraliser of the $\mathfrak{su}(2)$-copy corresponding to the maximal root is $\mathfrak{su}(6)$. By the previous argument, $\mathfrak{su}(6)$ has a $\mathfrak{d}_j$ summand with corresponding trivial $\mathfrak{f}_j$. \qedhere
\end{itemize}
\end{proof}

In Lemma \ref{Lem:technical} we showed that for each index $j = 1, \dots, m$ such that $\mathfrak{f}_j$ is trivial, there exists a left-invariant parallel subbundle $H_j \subset TM$, corresponding to the subalgebra $\mathfrak{h}_j = \langle e_1^j \rangle_\mathbb{H}$. Therefore, the number of trivial $\mathfrak{f}_j$ summands in the Joyce decomposition of $\g$ provides an indication of the extent to which the holonomy of the Obata connection reduces. This count is computed inductively using the procedure described in Theorem \ref{main}, and the results are summarised in Table \ref{table:table}. In the proof of Theorem \ref{main}, we show that for any Joyce hypercomplex manifold $ 
(M = \mathbb{T}^{\ell} \times G,I,J,K) $
appearing in the list \eqref{eqn:classification}, with the exception of $\mathrm{SU}(2n+1)$, the number of trivial $\f_j$ in the Joyce decomposition of $\g$ is at least $1$. 
\begin{table}[ht] \label{table:table}
\centering
%\small
\caption{Trivial $\mathfrak{f}_j$ summands in the Joyce decomposition of  $\mathfrak{g}$}
\vspace{0.5em}
\renewcommand{\arraystretch}{1.2}
\begin{tabular}{|c|c|}
\thickhline
\textbf{$G$} & \textbf{\# trivial $\mathfrak{f}_j$ summand} \\
\thickhline
$\mathrm{G}_2$                & 1 \\
\hline
$\mathrm{F}_4$                & 1 \\
\hline
$\mathrm{E}_6$                & 1 \\
\hline
$\mathrm{E}_7$                & 4 \\
\hline
$\mathrm{E}_8$                & 4 \\
\hline
$\mathrm{Sp}(k)$    & 1 \\
\hline
$\mathrm{SO}(2k{+}1)$ & $\lceil \frac{k}{2} \rceil $ \\
\hline
$\mathrm{SO}(4k)$   & $k{+}1$ \\
\hline
$\mathrm{SO}(4k{+}2) $ & $k$ \\
\hline
$\mathrm{SU}(2k)$   & 1 \\
\hline
$\mathrm{SU}(2k{+}1)$ & 0 \\
\thickhline
\end{tabular}
\end{table}

Before stating the next result, we recall a version of the Ambrose-Singer theorem adapted to the left-invariant setting:
 \begin{thm}[Alekseevski\u{i} {\cite[Proposition 2.1]{Al}}]\label{Thm:Alek}
Let $\nabla$ be a left-invariant affine connection on the Lie group $G$ with Lie algebra $\mathfrak{g}$. Then the holonomy algebra $\mathfrak{hol}(\nabla)$, based at the identity, is the smallest subalgebra of $\mathfrak{gl}(\mathfrak{g})$ containing the curvature endomorphisms $R(x, y)$, and closed under commutators with the left multiplication operators $\nabla_z : \mathfrak{g} \to \mathfrak{g}$ for any $x, y,z \in  \mathfrak{g}$.
\end{thm}
\begin{cor}
Let $M^{4n}$ and $\nabla$ be as in Theorem \ref{main}, and suppose that there are $j$ trivial $\mathfrak{f}_i$ summands in the Joyce decomposition of $\mathfrak{g}$. Then
\begin{equation*}
    \mathfrak{hol}(\nabla) \subseteq  \left\{ 
\begin{pmatrix}
a_{11}      & \cdots & a_{1,n-j}     & 0      & 0       & \cdots & 0 \\
\vdots      & \ddots & \vdots        & \vdots & \vdots  &  & \vdots \\
a_{n-j,1}   & \cdots & a_{n-j,n-j}   & 0      & 0    &  \cdots      & 0 \\
a_{n-j+1,1} & \cdots & a_{n-j+1,n-j} & a_{n-j+1,n-j+1} & 0 & \cdots & 0 \\
a_{n-j+2,1} & \cdots & a_{n-j+2,n-j}      & 0 & a_{n-j+2,n-j+2}    & \cdots & 0 \\
\vdots      & \      &  \vdots       & \vdots & \vdots    & \ddots & \vdots \\
a_{n,1}     & \cdots & a_{n,n-j}     & 0      & 0 & \cdots & a_{nn}
\end{pmatrix} \in \mathfrak{gl}(n,\H)
\right\}.
\end{equation*}
\end{cor}
\begin{proof}
From Lemma \ref{Lem:technical} we know that the $j$ copies of $\h_i\cong \mathbb{H}$ are preserved by $\nabla$. The result now follows from Theorem \ref{Thm:Alek},  where the matrix identification is obtained by choosing a basis of $\mathfrak{g}$ with $\mathfrak{h}_i$ spanning the last $j$ vectors.
\end{proof} 

\begin{rmk}
Note that Theorem \ref{main} can be slightly generalised to the case of $G$ semisimple. Let $M=\mathbb{T}^\ell \times G$ with $G$ semisimple and \[G \neq \prod_{i} \mathrm{SU}(2k_i+1).\] Then for any Joyce hypercomplex structure on $M$ the holonomy of the associated Obata connection is strictly contained in $\mathrm{GL}(n,\mathbb{H})$.
\end{rmk}

\subsection{The holonomy of \texorpdfstring{$\mathbb{T}^2 \times \mathrm{Sp}(2)$}{T2xSp2}}\label{Subsec:Sp2}

In this subsection, we compute explicitly the holonomy algebra of the Obata connection on $\mathbb{T}^2\times \mathrm{Sp}(2)$. This is the lowest possible dimensional case covered in Theorem \ref{main}.   

\medskip 
We begin by writing down an explicit description of the Joyce decomposition of $\mathfrak{sp}(2)$. The quaternionic matrices generating $\mathfrak{sp}(2)$ are given by:
\begin{gather*}
e^1_2=\begin{pmatrix}
i & 0 \\
0 & 0
\end{pmatrix}, \quad
e^1_3=\begin{pmatrix}
j & 0 \\
0 & 0
\end{pmatrix}, \quad
e^1_4=\begin{pmatrix}
k & 0 \\
0 & 0
\end{pmatrix}, \\
e^2_2=\begin{pmatrix}
0 & 0 \\
0 & i
\end{pmatrix}, \quad
e^2_3=\begin{pmatrix}
0 & 0 \\
0 & j
\end{pmatrix}, \quad
e^2_{4}=\begin{pmatrix}
0 & 0 \\
0 & k
\end{pmatrix}, \\
f^1_1=\begin{pmatrix}
0 & 1 \\
-1 & 0
\end{pmatrix}, \quad
f^1_2=\begin{pmatrix}
0 & i \\
i & 0
\end{pmatrix}, \quad
f^1_3=\begin{pmatrix}
0 & j \\
j & 0
\end{pmatrix}, \quad
f^1_4=\begin{pmatrix}
0 & k \\
k & 0
\end{pmatrix}. 
\end{gather*}
We denote by $\phi^i_j,\psi^k_l$ the associated dual basis i.e. $\phi^i_j(e_s^r)=\delta_{ir}\delta_{js}$, $\psi^k_l(f^r_s)=\delta_{kr}\delta_{ls}$. The structure equations can be computed as follows:
\begin{equation}
\begin{split}\label{strconst}
    d\phi^1_2 &= -2 \phi^1_3\wedge \phi^1_4 -2 \sigma_1\\
    d\phi^1_3 &= -2 \phi^1_4\wedge \phi^1_2 -2 \sigma_2\\
    d\phi^1_4 &= -2 \phi^1_2\wedge \phi^1_3 -2 \sigma_3\\
    d\phi^2_2 &= -2 \phi^2_3\wedge \phi^2_4 +2 \bar\sigma_1\\
    d\phi^2_3 &= -2 \phi^2_4\wedge \phi^2_2 +2 \bar \sigma_2\\
    d\phi^2_4 &= -2 \phi^2_2\wedge \phi^2_3 +2 \bar \sigma_3\\
    d\psi^1_1 &=\ \  \phi^1_2\wedge \psi_2^1 + \phi^1_3\wedge \psi^1_3+\phi^1_4\wedge \psi^1_4-\phi^2_2 \w \psi_2^1- \phi^2_3 \w \psi^1_3-\phi^2_4 \w \psi^1_4\\
    d\psi_2^1 &= -\phi^1_2\wedge \psi^1_1 - \phi^1_3\wedge \psi^1_4+\phi^1_4\wedge \psi^1_3+ \phi^2_2 \w \psi^1_1 + \phi^2_4 \w \psi^1_3-\phi^2_3 \w \psi^1_4 \\
    d\psi^1_3 &=\ \  \phi^1_2\wedge \psi^1_4 - \phi^1_3\wedge \psi^1_1-\phi^1_4\wedge \psi_2^1+ \phi^2_3 \w \psi^1_1 - \phi^2_4 \w \psi_2^1 +\phi^2_2 \w \psi^1_4 \\
    d\psi^1_4 &= -\phi^1_2\wedge \psi^1_3 + \phi^1_3\wedge \psi_2^1-\phi^1_4\wedge \psi^1_1+\phi^2_4 \w \psi^1_1  + \phi^2_3 \w \psi_2^1 -\phi^2_2 \w \psi^1_3  \\
\end{split}
\end{equation}
where
\begin{gather*}
    \sigma_1:=\psi^1_1\wedge \psi^1_2+\psi^1_3\wedge \psi^1_4, \quad \sigma_2:=\psi^1_1\wedge\psi^1_3+\psi^1_4\wedge \psi^1_2, \quad \sigma_3:=\psi^1_1\wedge\psi^1_4+\psi^1_2\wedge \psi^1_3, \\
    \bar{\sigma}_1:=\psi^1_1\wedge\psi^1_2-\psi^1_3\wedge \psi^1_4, \quad \bar{\sigma}_2:=\psi^1_1\wedge\psi^1_3-\psi^1_4\wedge \psi^1_2, \quad \bar{\sigma}_3:=\psi^1_1\wedge\psi^1_4-\psi^1_2\wedge \psi^1_3.
\end{gather*}
Fix any basis $\{e_1^1, e_1^2\}$ of $2\mathfrak{u}(1)\cong \R^2$. The Joyce decomposition is then given by
\begin{equation}\label{eqn:Joycesp2}
2\mathfrak{u}(1)  = \langle e_1^1, e_1^2 \rangle, \quad 
\mathfrak{d}_1  = \langle e_2^1, e_3^1, e_4^1 \rangle,\quad
\mathfrak{d}_2  = \langle e_2^2, e_3^2, e_{4}^2 \rangle,\quad
\mathfrak{f}_1  = \langle f_1^1, f_2^1, f_3^1, f_4^1 \rangle.
\end{equation}
Let $(I,J,K)$ be the Joyce hypercomplex structure corresponding to the decomposition \eqref{eqn:Joycesp2}. 
We may rewrite the above as
\begin{align} \label{eqn:decompos}
2\mathfrak{u}(1)\oplus \mathfrak{sp}(2) = 
\left( \langle e_1^1\rangle \oplus \di_1\right) \oplus \mathfrak{f}_1\oplus  \left(\langle e_1^2 \rangle \oplus \di_2\right) = \mathfrak{h}_1 \oplus \mathfrak{f}_1 \oplus \mathfrak{h}_2,
\end{align}
and by Theorem \ref{main} the left-invariant subbundle generated by $\h_2 $ is Obata parallel.

Let us now determine the holonomy algebra of the Obata connection \(\nabla\) associated to the hypercomplex structure \((I, J, K)\). Recall that by remark \ref{rmk:uniqueness}, for any other Joyce hypercomplex structure $(\tilde I, \tilde J, \tilde K)$ with Obata connection $\tilde \nabla$ one has that $\mathfrak{hol}(\nabla)\cong\mathfrak{hol}(\tilde \nabla)$. We should point out that the classification of irreducible holonomies of torsion-free affine connections is known, see \cite{MS}. This is for instance used in \cite{Sol} to determine the holonomy of the Obata connection on $\mathrm{SU}(3)$. However, in our case we cannot appeal to this classification since the holonomy is not irreducible; thus, to determine the holonomy algebra we have to compute it directly. \par
With respect to the global coframe $\{\phi^i_j,\psi^k_l\}$, we can identify $\nabla$ with the connection $1$-form:
\begin{equation} \label{eqn:theta}
\Theta = \scalebox{0.8}{ 
$\displaystyle
\left(
\begin{array}{cccc|cccc|cccc}
-\phi^1_1 & \phi^1_2 & \phi^1_3 & \phi^1_4 & \psi^1_1 & \psi^1_2 & \psi^1_3 & \psi^1_4 & 0 & 0 & 0 & 0 \\
-\phi^1_2 & -\phi^1_1 & -\phi^1_4 & \phi^1_3 & -\psi^1_2 & \psi^1_1 & -\psi^1_4 & \psi^1_3 & 0 & 0 & 0 & 0 \\
-\phi^1_3 & \phi^1_4 & -\phi^1_1 & -\phi^1_2 & -\psi^1_3 & \psi^1_4 & \psi^1_1 & -\psi^1_2 & 0 & 0 & 0 & 0 \\
-\phi^1_4 & -\phi^1_3 & \phi^1_2 & -\phi^1_1 & -\psi^1_4 & -\psi^1_3 & \psi^1_2 & \psi^1_1 & 0 & 0 & 0 & 0 \\ \hline
-\psi^1_1 & \psi^1_2 & \psi^1_3 & \psi^1_4 & -\phi^1_1 & \phi^2_2 & \phi^2_3 & \phi^2_4 & 0 & 0 & 0 & 0 \\
-\psi^1_2 & -\psi^1_1 & -\psi^1_4 & \psi^1_3 & -\phi^2_2 & -\phi^1_1 & -\phi^2_4 & \phi^2_3 & 0 & 0 & 0 & 0 \\
-\psi^1_3 & \psi^1_4 & -\psi^1_1 & -\psi^1_2 & -\phi^2_3 & \phi^2_4 & -\phi^1_1 & -\phi^2_2 & 0 & 0 & 0 & 0 \\
-\psi^1_4 & -\psi^1_3 & \psi^1_2 & -\psi^1_1 & -\phi^2_4 & -\phi^2_3 & \phi^2_2 & -\phi^1_1 & 0 & 0 & 0 & 0 \\ \hline
0 & 0 & 0 & 0 & -3\psi^1_1 & \psi^1_2 & \psi^1_3 & \psi^1_4 & -\phi^2_1 & \phi^2_2 & \phi^2_3 & \phi^2_4 \\
0 & 0 & 0 & 0 & -\psi^1_2 & -3 \psi^1_1 & -\psi^1_4 & \psi^1_3 & -\phi^2_2 & -\phi^2_1 & -\phi^2_4 & \phi^2_3 \\
0 & 0 & 0 & 0 & -\psi^1_3 & \psi^1_4 & -3\psi^1_1 & -\psi^1_2 & -\phi^2_3 & \phi^2_4 & -\phi^2_1 & -\phi^2_2 \\
0 & 0 & 0 & 0 & -\psi^1_4 & -\psi^1_3 & \psi^1_2 & -3\psi^1_1 & -\phi^2_4 & -\phi^2_3 & \phi^2_2 & -\phi^2_1
\end{array}
\right)
$}
\end{equation}
or equivalently, as the $\mathfrak{gl}(3,\H)$-valued $1$-form:
\[
\Theta =
\begin{pmatrix}
    -\phi^1_1+i\phi^1_2+j\phi^1_3+k\phi^1_4 & +\psi^1_1+i\psi^1_2+j\psi^1_3+k\psi^1_4 & 0 \\
    -\psi^1_1+i\psi^1_2+j\psi^1_3+k\psi^1_4 & -\phi^1_1+i\phi^2_2+j\phi^2_3+k\phi^2_4 & 0 \\
    0 & -3\psi^1_1+i\psi^1_2+j\psi^1_3+k\psi^1_4 & -\phi^2_1+i\phi^2_2+j\phi^2_3+k\phi^2_4
\end{pmatrix}.
\]
We compute the curvature form $F_{\Theta}:=d\Theta+ \Theta \w \Theta$ as
\begin{equation}
F_\Theta =
\begin{pmatrix}
    0 & 0 & 0 \\
    0 & 0 & 0 \\
    F_1 & F_2 & F_3
\end{pmatrix},\label{equ: curvature form}
\end{equation}
where
\begin{gather*}
    F_1 = -2i \bar \sigma_1 -2j\bar \sigma_2 -2k\bar \sigma_3,\\
    F_3 = +2i \bar \sigma_1 +2j\bar \sigma_2 +2k\bar \sigma_3,
\end{gather*}
and 
\begin{align*}
F_2 ={}& -3\big(\phi^1_1 \wedge \psi^1_1+ \phi^1_2 \wedge \psi^1_2 + \phi^1_3 \wedge \psi^1_3 + \phi^1_4 \wedge \psi^1_4  \\
      &\quad \quad   - \phi^2_1 \w \psi^1_1 -\phi^2_2 \w  \psi^1_2 - \phi^2_3 \w \psi^1_3 - \phi^2_4 \w \psi^1_4 \big) \\
    & + i\big(\phi^1_1 \wedge \psi^1_2 + \phi^2_3 \wedge \psi^1_4 -\phi^2_4 \w \psi^1_3 
        - \phi^1_2 \wedge \psi^1_1 - \phi^1_3 \wedge \psi^1_4 \\
    &\quad \quad + \phi^1_4 \wedge \psi^1_3 + \phi^2_2 \w \psi^1_1 - \phi^2_1 \w \psi^1_2 \big) \\
    & + j\big(\phi^1_2 \wedge \psi^1_4 - \phi^1_3 \wedge \psi^1_1 - \phi^1_4 \wedge \psi^1_2 
        + \phi^2_3 \w \psi^1_1 + \phi^2_4 \w \psi^1_2 \\
    &\quad \quad + \phi^1_1 \wedge \psi^1_3 -\phi^2_2 \w  \psi^1_4 - \phi^2_1 \w \psi^1_3 \big) \\
    & + k\big(-\phi^2_3 \w \psi^1_2 + \phi^1_1 \wedge \psi^1_4 + \phi^2_2 \w \psi^1_3 
        - \phi^1_2 \wedge \psi^1_3 + \phi^1_3 \wedge \psi^1_2 \\
    &\quad \quad - \phi^1_4 \wedge \psi^1_1 + \phi^2_4 \w \psi^1_1 - \phi^2_1 \w \psi^1_4 \big)
\end{align*}

Therefore, we see that $F_\Theta(X,Y)$ has $7$ generators, which, expressed as tensors, are given by 
\begin{equation}
\begin{split}\label{equ: curvature endomorphisms}
    \tau_1&=(\phi^1_2 - \phi^2_2) \otimes e^2_1+(\phi^2_1-\phi^1_1) \otimes e^2_2+( \phi^2_4-\phi^1_4) \otimes e^2_3+(\phi^1_3 - \phi^2_3) \otimes e^2_4, \\
    \tau_2&=(\phi^1_3 - \phi^2_3) \otimes e^2_1+(\phi^1_4 - \phi^2_4) \otimes e^2_2+(\phi^2_1-\phi^1_1) \otimes e^2_3+( \phi^2_2-\phi^1_2 ) \otimes e^2_4, \\
    \tau_3&=(\phi^1_4 - \phi^2_4) \otimes e^2_1+(\phi^2_3-\phi^1_3 ) \otimes e^2_2+(\phi^1_2 - \phi^2_2) \otimes e^2_3+(\phi^2_1-\phi^1_1) \otimes e^2_4, \\
    \tau_4&=\psi^1_1 \otimes e^2_1+\psi^1_2 \otimes e^2_2+\psi^1_3 \otimes e^2_3+\psi^1_4 \otimes e^2_4, \\
    \tau_5&=\psi^1_2 \otimes e^2_1-\psi^1_1 \otimes e^2_2-\psi^1_4 \otimes e^2_3+\psi^1_3 \otimes e^2_4,\\
    \tau_6&=\psi^1_3 \otimes e^2_1+\psi^1_4 \otimes e^2_2-\psi^1_1 \otimes e^2_3-\psi^1_2 \otimes e^2_4,\\
    \tau_7&=\psi^1_4 \otimes e^2_1-\psi^1_3 \otimes e^2_2+\psi^1_2 \otimes e^2_3-\psi^1_1 \otimes e^2_4.
\end{split}
\end{equation}
By the Ambrose-Singer theorem (see Theorem \ref{Thm:Alek}), in order to compute the full holonomy algebra of the Obata connection we also need to compute $(\nabla_XF_\Theta)(Y,Z)$ and its higher derivatives. 
First note that we can write 
$$F_\Theta=\alpha_i \otimes \tau_i,$$ 
where $\alpha_i$ correspond to the $2$-forms appearing in $F_1, F_2, F_3$. If we now compute 
$$(\nabla_X F_\Theta)(Y,Z)=(\nabla_X\alpha_i)(Y,Z)  \tau_i+\alpha_i(Y,Z)\nabla_X \tau_i,$$ then we see that we recover the endomorphisms in (\ref{equ: curvature endomorphisms}) and obtain possibly new endomorphisms from the terms of the form $\alpha_i(Y,Z) \nabla_X\tau_i$. Thus, it only suffices to consider those.\par
On the other hand, from Theorem \ref{main} we see that $\nabla_X\mathfrak{h}_2$ lies in $\mathfrak{h}_2$. It follows that the holonomy algebra has to be generated by $\H$-matrices of the form
\begin{equation}
\begin{pmatrix}\label{equ: general matrices}
  0 & 0 & 0\\
  0 & 0 & 0\\
  q_1 & q_2 & q_3
\end{pmatrix}, \qquad q_i\in \H\,.
\end{equation}
This is a Lie algebra of real dimension 12. The crucial point is to see if the holonomy algebra is a strict sub-algebra of the latter. To this end we begin by determining the span of $(\nabla_XF_\Theta)(Y,Z)$.
\begin{lem}\label{lem: new generators}
We have the following:
\begin{equation}
\begin{split}
    \nu_1:=\nabla_{e^1_1} \tau_1&=+\phi^1_2\otimes e^2_1 - \phi^1_1 \otimes e^2_2 -\phi^1_4\otimes e^2_3+\phi^1_3 \otimes e^2_4, \\
    \nu_2:=\nabla_{e^1_2} \tau_1&=+\phi^1_1\otimes e^2_1 + \phi^1_2 \otimes e^2_2 +\phi^1_3\otimes e^2_3+\phi^1_4 \otimes e^2_4,\\
    \nu_3:=\nabla_{e^1_3} \tau_1&=-\phi^1_4\otimes e^2_1 + \phi^1_3 \otimes e^2_2 -\phi^1_2\otimes e^2_3+\phi^1_1 \otimes e^2_4,\\
    \nu_4:=\nabla_{e^1_{4}} \tau_1&=+\phi^1_3\otimes e^2_1 + \phi^1_4 \otimes e^2_2 -\phi^1_1\otimes e^2_3-\phi^1_2 \otimes e^2_4.\\
\end{split}
\end{equation}
\end{lem}
\begin{proof}
This follows by a direct computation using the connection form \eqref{eqn:theta}.  
\end{proof}
\begin{thm}\label{thm: sp2T2}
    The holonomy algebra $\mathfrak{hol}(\nabla)$ is the $11$-dimensional Lie algebra:
    \begin{equation}\label{equ: full hol algebra}
    \mathfrak{hol}(\nabla)\cong \Bigg\{ \begin{pmatrix}
        0 & 0 & 0 \\
        0 & 0 & 0 \\
        q_1 & q_2 & p \\ 
    \end{pmatrix}\in \mathfrak{gl}(3,\H) \ | \ q_1,q_2 \in \H, \ p \in \operatorname{Im}(\H)\Bigg\}.
    \end{equation} 
\end{thm}
\begin{proof}
From Lemma \ref{lem: new generators}, the endomorphisms $\nu_1,\dots,\nu_4$ are linearly independent. It is also easy to see that they are linearly independent of $\tau_1,\dots,\tau_7$ as well.
Thus, this yields a total of 11 generators for the holonomy algebra:
\[
\langle \tau_1,\dots,\tau_7, \nu_1,\dots \nu_4\rangle \subset \mathfrak{hol}(\nabla).
\]
Since $11=\dim (\langle \tau_1,\dots,\tau_7, \nu_1,\dots \nu_4\rangle) \le \dim(\mathfrak{hol}(\nabla)) \le 12$, to prove the result we only need to show that the holonomy algebra cannot be $12$-dimensional. One can show this by either directly computing $\nabla^2 F_{\Theta}$, or one can instead appeal to Theorem \ref{main2} below which asserts that the holonomy algebra of $\nabla$ is necessarily a sub-algebra of $\mathfrak{sl}(3,\H)$ i.e. $q_3$ in (\ref{equ: general matrices}) has to be purely imaginary. This concludes the proof.
\end{proof}
It is worth pointing out that unlike for metric connections, if there is a non-trivial subbundle preserved by the Obata connection then this does not imply the existence of a complementary parallel subbundle. This can indeed be seen in the above example.

\subsection{The case of \texorpdfstring{$\SU(2n+1)$}{SU(2n+1)}}
In this section, we focus on $\SU(2n+1)$ for $n > 1$, which is the only case excluded by Theorem \ref{main}. This case is particularly interesting since unlike the situations covered by Theorem~\ref{main}, where the holonomy reduction holds regardless of the chosen Joyce hypercomplex structure, in the $\SU(2n+1)$ case the choice of hypercomplex structure plays a crucial role. We elaborate below. \par
The Joyce decomposition of $\mathfrak{su}(2n+1)$ can be constructed iteratively via the relation
\[
\mathfrak{su}(2n+1) = \mathfrak{su}(2n-1) \oplus \mathfrak{u}(1) \oplus \mathfrak{d}_1 \oplus \mathfrak{f}_1.
\]
In particular, for any $k \leq n$, the Lie algebra $\mathfrak{su}(2k+1)$ is embedded as a subalgebra of $\mathfrak{su}(2n+1)$. This leads to the chain of inclusions
\[
\mathfrak{su}(3) \subset \mathfrak{su}(5) \subset \dots \subset \mathfrak{su}(2n+1).
\]
Note that, the $\mathfrak{su}(2k+1)$ is not in general invariant under the action of $(I, J, K)$ of $\mathfrak{su}(2n+1)$. When this is the case, i.e. for appropriate choices of $(I,J,K)$, we shall show that this $\mathfrak{su}(2k+1)$ is indeed Obata invariant. When $n=1$ there are no such inclusion which is why we consider $n>1$.

\par
Proceeding iteratively, we obtain the decomposition
\[
\mathfrak{su}(2n+1) = \mathfrak{b} \oplus \bigoplus_{i=1}^n \mathfrak{d}_i \oplus \bigoplus_{i=1}^n \mathfrak{f}_i,
\]
where $\mathfrak{f}_i = (2n - (2i - 1))\mathbb{C}^2$. The Lie algebra $\mathfrak{su}(2n+1)$ consists of $(2n+1) \times (2n+1)$ skew-Hermitian, trace-free matrices. Such matrices can be written in block form as
\begin{equation} \label{eqn:su2n+1}
\left(
\begin{array}{c|c|cc}
D_1 & \multicolumn{3}{c}{F_1} \\
\hline
\multirow{4}{*}{$-\bar F_1^t$} & D_2 & \multicolumn{2}{c}{F_2} \\
\cline{2-4}
                     & \multirow{3}{*}{$-\bar F_2^t$} & \ddots & \vdots \\
                     &     & \cdots & 
    \begin{array}{|c|c}
     \hline
        D_n & F_n \\
        \hline
        -\bar F_n^t & B
    \end{array}
\end{array}
\right),
\end{equation}
where each $D_i \in \mathfrak{u}(2)$, each $F_i$ is a matrix in $\mathcal{M}(2 \times (2n - (2i - 1)),\C)$, and $B \in \mathbb{C}$ is such that $\sum_i \mathrm{tr}(D_i) + B = 0$. The subspace $\mathfrak{d}_i$ consists of matrices in which only $D_i$ is non-zero, $\mathfrak{f}_i$ consists of those where only $F_i$ is non-zero, and $\mathfrak{b}$ consists of trace-free diagonal matrices that commute with every $D_i$. \par
With this block description, it is clear that the copy of $\mathfrak{su}(2k+1)$ inside $\mathfrak{su}(2n+1)$ is embedded in the bottom-right square block of size $2k+1$. \par
Since $\dim(\bi) = n$, there are up to $n^2$ parameters of hypercomplex structures determined by the choice of isomorphism $\bi \cong \mathbb{R}^n$. We will show that, for some of these hypercomplex structures, the associated Obata connection preserves a left-invariant subbundle. \par
For the sake of simplicity, let us first consider the case of $\mathrm{SU}(5)$. From the argument in the proof of Theorem~\ref{main}, we have
\[
\mathfrak{su}(5) = \mathfrak{u}(1) \oplus \mathfrak{su}(3) \oplus \mathfrak{d}_1 \oplus \mathfrak{f}_1,
\]
and the Joyce decomposition is obtained by further decomposing
\[
\mathfrak{su}(3) = \mathfrak{u}(1) \oplus \mathfrak{d}_2 \oplus \mathfrak{f}_2.
\]
In this case, the subspace $\mathfrak{b}$ has dimension $2$, and one can choose a basis $\{E_1,E_2\}$ of $\mathfrak{b}$ such that $E_2$ generates the $\mathfrak{u}(1)$-copy inside $\mathfrak{su}(3)$. More precisely:
\[
\bi = \left \langle 
E_1:=\left(\begin{array}{cc|ccc}
    3i & 0 & 0 & 0 & 0\\ 
    0 & 3i & 0 & 0 & 0\\ \hline
    0 & 0 & -2i & 0 & 0\\
    0 & 0 & 0 & -2i & 0\\
    0 & 0 & 0 & 0 & -2i
\end{array}\right), \ 
E_2:=\left(\begin{array}{cc|ccc}
    0 & 0 & 0 & 0 & 0\\
    0 & 0 & 0 & 0 & 0\\ \hline
    0 & 0 & i & 0 & 0\\
    0 & 0 & 0 & i & 0\\
    0 & 0 & 0 & 0 & -2i
\end{array} \right)
\right \rangle 
\]
Thus, the Joyce hypercomplex structures depend on $4$ parameters corresponding  to the choice a basis $\{e_1^1,e_1^2\}$ (see Section \ref{sec:hcx} for the notation). Explicitly, we can identify the hypercomplex structures with the matrix $A=(\begin{smallmatrix} a_1 & a_3 \\ a_2 & a_4 \end{smallmatrix})\in \mathrm{GL}(2,\R)$, where $e_1^1= a_1 E_1 + a_2 E_2$, \ $e^2_1= a_3 E_1 + a_4 E_2$.
\begin{thm} \label{thm:su5}
    Consider $\SU(5)$ with the Joyce hypercomplex structure $(I,J,K)$ corresponding to the matrix 
    \[
    A=\begin{pmatrix}
        a_1 & 0 \\
        a_2 & a_4 \\
    \end{pmatrix} \in \mathrm{GL}(2,\R), 
    \]
    i.e., 
    \[
    \mathfrak{su}(5)=\Big( \langle a_1 E_1+a_2 E_2 \rangle \oplus \mathfrak{d}_1 \oplus \mathfrak{f}_1 \Big) \oplus \Big( \langle a_4 E_2 \rangle \oplus \mathfrak{d}_2 \oplus \mathfrak{f}_2\Big).
    \]
    Then the hypercomplex subspace $\mathfrak{su}(3)= \langle a_4 E_2 \rangle \oplus \mathfrak{d}_2 \oplus \mathfrak{f}_2$ is invariant by the Obata connection $\nabla$ of $(I,J,K)$.
    In particular, the holonomy group $\mathrm{Hol}(\nabla) \subsetneqq \mathrm{GL}(6,\mathbb{H})$.
\end{thm}
\begin{proof}
    Let $X \in \mathfrak{su}(5)$ and $Y \in \mathfrak{su}(3)$, we want to show that $\nabla_X Y \in \mathfrak{su}(3)$. \par
    From the choice of the basis $\{E_1,E_2\}$, and using that $[\mathfrak{f}_2,\mathfrak{f}_2] \subset \mathfrak{su}(3)$, we immediately get that $\mathfrak{su}(3)$ is an hypercomplex sub-algebra. Consequently, if $X \in \mathfrak{su}(3)$ then the result is clear from  the expression of the Obata connection, see \eqref{equ: obata connection sol}. 

    Suppose now $X=e_1^1=a_1 E_1+a_2 E_2$, then 
    \begin{align*}
        \nabla_{e_1^1} Y 
        &= \nabla_Y e_1^1 + [e_1^1, Y] \\
        &= [e_1^1, Y ] \in \mathfrak{f}_2 \subset \mathfrak{su}(3)
    \end{align*}
    where we used Lemma \ref{lem:nablae1}, \eqref{eqn:joyce1} and Lemma \ref{Lem:structure}.

    Suppose now $X \in \mathfrak{d}_1$. Again from Lemma \ref{lem:nablae1}, $\nabla_X e_1^2=0$ and so $\nabla_X \mathfrak{d}_2=0$. For the remaining case, we have
    \begin{align*}
     2   \nabla_X \mathfrak{f}_2 &= [X,\mathfrak{f}_2]+ I[IX, \mathfrak{f}_2]-J[X,J\mathfrak{f}_2]+K[IX,J\mathfrak{f}_2] \in \mathfrak{f}_2 \subset \mathfrak{su}(3),
    \end{align*}
    since $[X,\mathfrak{f}_2],[X,J\mathfrak{f}_2]$ vanish from \eqref{eqn:joyce3} and $[IX, \mathfrak{f}_2], [IX,J\mathfrak{f}_2]$ lie in $\f_2$ since $IX$ belongs to $\langle e_1^1\rangle \oplus \di_1$. 

    Lastly, suppose that $X \in \mathfrak{f}_1$. As before from Lemma \ref{lem:nablae1} we have $\nabla_X e_1^2 = 0$, and hence $\nabla_X \mathfrak{d}_2 = 0$. It remains to compute $\nabla_X \mathfrak{f}_2$. Let $Z_1 \in \mathfrak{f}_1$ and $Z_2 \in \mathfrak{f}_2$ then
    \begin{align*}
        [Z_1,Z_2]+I[IZ_1,Z_2] &= [Z_1,Z_2]+I[[e_2^1,Z_1],Z_2]\\
        &=[Z_1,Z_2]-I[[Z_1,Z_2],e_2^1]- I[[Z_2,e_2^1],Z_1]\\
        &= [Z_1,Z_2]-I(-I[Z_1,Z_2])\\
        &=0
    \end{align*}
     where we used \eqref{eqn:joyce3} and that $[Z_1,Z_2] \in \mathfrak{f}_1$ from Lemma \ref{Lem:structure}. Comparing with expression \eqref{equ: obata connection sol}, and using that $\f_2$ is $J$-invariant, we see that $\nabla_X \mathfrak{f}_2=0$.
\end{proof}
Let us now consider the Lie algebra $\mathfrak{su}(2n+1)$, $n > 1$, and let us fix a basis ${E_1, \dots, E_n}$ of the abelian subalgebra $\bi \cong \mathbb{R}^n$ such that each element $E_k$ belongs to the subalgebra $\mathfrak{su}(2n + 3 - 2k) \subset \mathfrak{su}(2n+1)$. For instance, $E_1 \in \mathfrak{su}(2n+1)$, $E_2 \in \mathfrak{su}(2n-1)$, and $E_n \in \mathfrak{su}(3)$. \par
From the matrix point of view \big(see \eqref{eqn:su2n+1}\big), each $E_k=\mathrm{diag}(a_1,a_2,\dots,a_{2n+1})$ is a diagonal, trace-free matrix that commutes with every block $D_j$, and satisfies $a_{i} = 0$ for all $i = 1, \dots, 2k-2$. We generalise Theorem \ref{thm:su5} as follows:
\begin{thm} \label{thm:su2n+1}
    Fix $ n>1$ and an integer $k\in \{1,\dots,n-1\}$. Consider the Lie group $\SU(2n+1)$ with the Joyce hypercomplex structure $(I,J,K)$ corresponding to the matrix 
    \[
A = 
\left(
\begin{array}{ccc|ccc}
a_{1,1} & \cdots & a_{1,n-k} &  0 & \cdots & 0\\
\vdots & \ddots & \vdots &  \vdots & \ddots & \vdots \\
a_{n-k,1} & \cdots & a_{n-k,n-k} & 0 & \cdots & 0\\
\hline
a_{n-k+1,1} & \cdots & a_{n-k+1,n-k} & a_{n-k+1,n-k+1} &\cdots &a_{n-k+1,n} \\
\vdots & \ddots & \vdots & \vdots & \ddots & \vdots\\
a_{n,1} & \cdots & a_{n,n-k} & a_{n,n-k+1} & \cdots & a_{n,n}
\end{array}
\right)\in \mathrm{GL}(n,\R), 
    \]
    i.e., 
    \[
    \mathfrak{su}(2n+1)= \bigoplus_{i=1}^{n-k}\Big( \langle e_1^i \rangle \oplus \mathfrak{d}_i \oplus \mathfrak{f}_i \Big) \oplus \bigoplus_{i=n-k+1}^{n} \Big( \langle e_1^i \rangle \oplus \mathfrak{d}_i \oplus \mathfrak{f}_i\Big),
    \]
    where for $i=1,\dots,n-k$:
    \[
     \ e^i_1=\sum_{j=1}^n a_{ji}E_j,  
    \]
    and for $i=n-k,\dots,n$:
    \[
     e^i_1=\sum_{j=n-k+1}^n a_{ji}E_j \in \mathfrak{su}(2k+1).
    \]
    Then the hypercomplex subspace \[\mathfrak{su}(2k+1)=\bigoplus_{i=n-k+1}^{n} \Big( \langle e_1^i \rangle \oplus \mathfrak{d}_i \oplus \mathfrak{f}_i\Big)\] is invariant by the Obata connection $\nabla$ of $(I,J,K)$.
    In particular, the holonomy group $\mathrm{Hol}(\nabla) \subsetneqq \mathrm{GL}(n(n+1),\mathbb{H})$.
\end{thm}
\begin{proof}
The proof is analogous to that of Theorem~\ref{thm:su5}, and uses crucially the fact that, with respect to the hypercomplex structure $(I, J, K)$, the subalgebra $\mathfrak{su}(2k+1)$ is hypercomplex.
\end{proof}
Combining Theorems \ref{main}, \ref{thm:su2n+1} and the fact that $S^1 \times \SU(2)$ is Obata flat, we get:
\begin{cor}
    Every manifold in the list $\eqref{eqn:classification}$, aside from $\mathrm{SU}(3)$, admits a Joyce hypercomplex structure with Obata holonomy \textit{strictly} contained in $\mathrm{GL}(n,\mathbb{H})$. 
\end{cor}
Theorem \ref{thm:su2n+1} shows that there are infinitely many left-invariant hypercomplex structures on \(\mathrm{SU}(2n+1)\) for which the Obata connection has reduced holonomy. However, this does not cover all the possible cases. For simplicity, let us consider the case of \(\mathrm{SU}(5)\) in more detail. 

Recall that the Joyce hypercomplex structures on \(\mathrm{SU}(5)\) are parametrised by matrices $(\begin{smallmatrix} a_1 & a_3 \\ a_2 & a_4 \end{smallmatrix})$. In Theorem~\ref{thm:su5}, we showed that when $a_3=0$ the Obata connection preserves a left-invariant subbundle of \(T\mathrm{SU}(5)\). In the complementary case when \(a_3 \neq 0\), the subalgebra \(\mathfrak{su}(3)\) is no longer hypercomplex. Although a full characterisation of this case remains open, we conjecture that the Obata connection in this case has holonomy \textit{equal} to \(\mathrm{GL}(6, \mathbb{H})\). We expect that a similar behaviour also holds for $\SU(2n+1)$, with $ n > 1$.  As a partial evidence to our claim, we show:
\begin{thm}\label{thm: su5}
    Consider $\SU(5)$ with the Joyce hypercomplex structure $(I,J,K)$ corresponding to the matrix 
    \[
    A=\begin{pmatrix}
        0 & 1 \\
        1 & 0 \\
    \end{pmatrix} \in \mathrm{GL}(2,\R), 
    \]
    i.e., 
    \[
    \mathfrak{su}(5)=\Big( \langle E_2 \rangle \oplus \mathfrak{d}_1 \oplus \mathfrak{f}_1 \Big) \oplus \Big( \langle  E_1 \rangle \oplus \mathfrak{d}_2 \oplus \mathfrak{f}_2\Big).
    \]
    Then the Obata holonomy group $\mathrm{Hol}(\nabla)$ is isomorphic to $\mathrm{GL}(6,\mathbb{H})$.
\end{thm}
\begin{proof}
    Denote by $R$ the curvature tensor of the Obata connection (\ref{equ: obata connection sol}) and
    let $\nabla^k\mathcal{R}$ denote the space of $k^{th}$ covariant derivative of the curvature tensor i.e.
    \[
    \nabla^k\mathcal{R}:=\{\nabla^k_{x_1,\cdots,x_k}R_{x,y}\ |\ \forall \ x,y,x_i \in \mathfrak{su}(5)\}.
    \]
    Observe that $R$ and its derivatives can be computed using only the Lie bracket and the hypercomplex structures $I,J,K$ (which are themselves determined by the Lie bracket after fixing $\mathcal{B}$, see Section \ref{sec:hcx}). \par
    With the aid of MAPLE software, we  
    have been able to determine the dimension of the latter spaces as follows: 
    \begin{equation}
    \mathrm{dim}(\mathcal{R})=52, \quad \mathrm{dim}(\{\mathcal{R},\nabla \mathcal{R}\})=138, \quad
    \mathrm{dim}(\{\mathcal{R},\nabla\mathcal{R},\nabla^2\mathcal{R}\})=144.
    \end{equation}
    Since $\dim(\mathfrak{gl}(6,\mathbb{H}))=144$ and we know that $\mathfrak{hol}(\nabla) \subset \mathfrak{gl}(6,\mathbb{H})$, the result follows from Theorem \ref{Thm:Alek}, the classification of the holonomy groups of irreducible torsion free affine connections in \cite{MS} and the fact that $\mathrm{SU}(5)$ is simply connected.
\end{proof}
Note that our proof relies on directly determining the dimension of the holonomy algebra with the help of the computer; this is typically unfeasible in higher dimensions. In \cite{Sol}, Soldatenkov was able to prove the irreducibility of the holonomy group in the $\mathrm{SU}(3)$ case without having to determine the rank of the holonomy algebra. Unfortunately, we have been unable to generalise his argument to the $\mathrm{SU}(2n+1)$ case.
It is worth mentioning, however, that we were able to check that one needs to compute up to $\nabla^4 \mathcal{R}$ in the $\SU(3)$ case in order to generate the full holonomy algebra of $\mathfrak{gl}(2,\mathbb{H})$; this provides a more direct, but computation-heavy, proof of the result in \cite{Sol}. It would be interesting to have a neater argument which is applicable to all the cases excluded by Theorem \ref{thm:su2n+1}. \par
\medskip

As a consequence of Theorems \ref{thm:su2n+1} and \ref{thm: su5}, it follows that the irreducibility of the Obata holonomy group is not preserved under variations of the hypercomplex structure, even within the Joyce parameter space $\mathrm{GL}(2,\R)$. For instance, one can deform $A_0=(\begin{smallmatrix}
   0 & 1 \\
   1 & 0 \\
\end{smallmatrix})$ to $A_1=(\begin{smallmatrix}
   1 & 0 \\
   2 & -1 \\
\end{smallmatrix})$ via the curve
\[
A_t=\begin{pmatrix}
   t & 1-t \\
   1+t & -t \\
\end{pmatrix} \in \mathrm{GL}(2,\R) , \ \ t\in (-\epsilon,1+\epsilon).
\]
We can summarise the above observation into:
\iffalse
Recall that a left-invariant hypercomplex structures on $\mathrm{SU}(5)$ can be identified with a matrix $A \in \mathrm{GL}(2,\R)$, as described at the beginning of this section. Consequently, a \emph{curve} of left-invariant hypercomplex structures corresponds to a \emph{curve} of matrices in $\mathrm{GL}(2,\R)$. \par
Consider the following curve of left-invariant hypercomplex structures $(I_t,J_t,K_t)$ determined by the curve $A_t$ of matrices
\[
A_t=\begin{pmatrix}
   t & 1-t \\
   1+t & -t \\
\end{pmatrix}, \ \ t\in (-\epsilon,1+\epsilon).
\]
Observe that $A_0=\begin{pmatrix}
   0 & 1 \\
   1 & 0 \\
\end{pmatrix}$and hence, by Theorem~\ref{thm: su5}, the holonomy of the Obata connection associated to $(I_0,J_0,K_0)$ is $\mathrm{GL}(6,\H)$. On the other hand, $A_1=\begin{pmatrix}
   1 & 0 \\
   2 & -1 \\
\end{pmatrix}$, implying that the holonomy of the Obata connection of the hypercomplex structure $(I_1,J_1,K_1)$ lies strictly inside $\mathrm{GL}(6,\H)$, by Theorem \ref{thm:su5}. \par
\fi
\begin{cor}
The irreducibility of the holonomy of the Obata connection is not preserved under deformations of the hypercomplex structure. 
\end{cor}

\section{Proof of Theorem \ref{THM:B}}\label{Sec:restrhol}
In this section, we investigate Joyce hypercomplex manifolds whose Obata connection has restricted holonomy group contained in $\mathrm{SL}(n,\mathbb{H})$.

For a hypercomplex manifold, the restricted holonomy lies in $\mathrm{SL}(n,\mathbb{H})$ if and only if the Ricci tensor of the Obata connection vanishes, see \cite[Section 6]{Sal86} and \cite[Theorem 5.6]{AM}. In view of this, we begin by determining an explicit formula of the Obata-Ricci tensor on a hypercomplex Lie algebra.

\begin{lem}\label{lem:Obata1form}
Let $\g $ be a Lie algebra of real dimension $4n$ equipped with a hypercomplex structure $(I,J,K)$. Then the Obata connection $1$-form induced on the canonical line bundle $\Lambda^{2n,0}_I(\mathfrak{g})$ is given by
\[
\eta(X)=-\frac12 \mathrm{tr}(\mathrm{ad}_X)-\frac12\mathrm{tr}(L\mathrm{ad}_{LX})\,, \qquad X\in \g\,,
\]
for any $L\in \{I,J,K\}$. In particular, the Obata-Ricci tensor is
\[
\mathrm{Ric}^{\mathrm{Ob}}(X,Y)= \frac12\mathrm{tr}(\mathrm{ad}_{[X,Y]})+\frac12\mathrm{tr}(L\mathrm{ad}_{L[X,Y]})\,, \qquad X,Y\in \g\,.
\]
\end{lem}
\begin{proof}
Since $\eta$ does not depend on the chosen complex structure, it suffices to prove the Lemma for $L=I$.\par 
Let $\alpha$ denote the $(1,0)$ part of $\eta$ with respect to $I$. Then, according to \cite[Lemma 3.3]{FG}, $\bar \alpha = \Lambda_\Omega(\bar \partial \Omega)$, where
\[
\Omega=\frac{\omega_J+i\omega_K}{2}
\]
is the $(2,0)$ form associated to a given hyperhermitian metric $g$ compatible with $(I,J,K)$. \par
Fix $W\in \g^{1,0}$ and an orthonormal basis $\{e_1,\dots,e_{4n}\}$ adapted to $(I,J,K)$. Set $Z_i=e_{2i-1}-iIe_{2i-1}=e_{2i-1}-ie_{2i}$, for $i=1,\dots,2n$. With respect to this basis we have
\[
\Omega= \frac{1}{2}\sum_{i=1}^n Z^{2i-1} \wedge Z^{2i}\,. 
\]
We compute
\[
\begin{split}
\bar \alpha(\bar W)&=(\Lambda\bar \partial \Omega)(\bar W)\\
&=\frac{1}{2}\sum_{i=1}^{n}\bar \partial \Omega(\bar W,Z_{2i-1},Z_{2i})\\
&=\frac{1}{2}\sum_{i=1}^{n}\left( -\Omega([\bar W,Z_{2i-1}],Z_{2i}) +\Omega([\bar W,Z_{2i}],Z_{2i-1})-\Omega([Z_{2i-1}, Z_{2i}],\bar W)\right).
\end{split}
\]
Since $\Omega$ is of type $(2,0)$ the last term vanishes, yielding
\[
\begin{split}
\bar \alpha(\bar W)&=\frac{1}{2}\sum_{i=1}^{n} \left( -\Omega([\bar W,Z_{2i-1}],Z_{2i}) +\Omega([\bar W,Z_{2i}],Z_{2i-1})\right)\\
&=-\frac{1}{2}\sum_{j=1}^{2n} \Omega([\bar W,Z_{j}],J\bar Z_{j})\\
&=-\frac{1}{2}\sum_{j=1}^{2n} g([\bar W,Z_{j}],\bar Z_{j})\,.
\end{split}
\]
In terms of the real basis we rewrite this as
\[
\begin{split}
\bar \alpha(\bar W)&=-\frac12\sum_{j=1}^{2n} g([\bar W,e_{2j-1}-ie_{2j}],e_{2j-1}+ie_{2j})\\
&=-\frac12\sum_{j=1}^{2n} \left( g([\bar W,e_{2j-1}],e_{2j-1}) +g([\bar W,e_{2j}],e_{2j})-i g([\bar W,e_{2j}],e_{2j-1})+ig([\bar W,e_{2j-1}],e_{2j})\right)\\
&=-\frac12\sum_{k=1}^{4n}g([\bar W,e_k],e_k)-\frac{i}{2}\sum_{k=1}^{4n}g([\bar W,e_{k}],Ie_{k})\\
&=-\frac12\mathrm{tr}(\mathrm{ad}_{\bar W})+\frac{i}{2}\,\mathrm{tr}(I\mathrm{ad}_{\bar W})\,.
\end{split}
\]
We conclude that for every $X\in \g$
\[
\eta(X)=2\Re \left(\bar \alpha(X^{0,1}) \right)=-\frac12 \mathrm{tr}(\mathrm{ad}_X)-\frac12\mathrm{tr}(I\mathrm{ad}_{IX})\,.
\]
Finally, from \cite[Proposition 3.4 (d)]{FG} we have  
\[
\mathrm{Ric}^{\mathrm{Ob}}(X,Y)=d\eta(X,Y)=\frac12\mathrm{tr}(\mathrm{ad}_{[X,Y]})+\frac12\mathrm{tr}(I\mathrm{ad}_{I[X,Y]})
\]
as claimed.
\end{proof}

We now compute the Obata connection $1$-form on Joyce hypercomplex manifolds.

\begin{prop}\label{Prop:Lee}
Let $G$ be a compact simple Lie group and $(I,J,K)$ a Joyce hypercomplex structure on $\mathbb{T}^{\ell}\times G$. 
Then the Obata connection $1$-form $\eta$ on $\Lambda^{2n,0}_I(\ell \mathfrak{u}(1) \oplus \g)$ takes the following expression
\[
\eta(X)= 2\sum_{j=1}^m \frac{1}{\lambda_j^2} \left(1+\dim_{\H}(\mathfrak{f}_j)\right)h(e_3^j, JX),
\]
where $h$ is a bi-invariant metric, and $\lambda_j^2=h(e_3^j,e_3^j)$.
\end{prop}
\begin{proof}
From Lemma \ref{lem:Obata1form}, using the unimodularity of $\mathbb{T}^{\ell}\times G$, we obtain
\begin{align*}
\eta(X)=-\frac12\mathrm{tr}(J\mathrm{ad}_{JX})=-\frac12\mathrm{tr}(\mathrm{ad}_{JX}J)= -\frac12\sum_{i=1}^{4n} h([JX,Je_i],e_i)=\frac{1}{2}\sum_{i=1}^{4n}h(JX,[e_i,Je_i])\,,
\end{align*}
where $\{e_1,\dots,e_{4n}\}$ is an orthonormal basis with respect to $h$. \par
Since $h$ is a bi-invariant metric on $\mathbb{T}^{\ell}\times G$, $h_{| \g \times \g}$ is a multiple of the Killing-Cartan form on $\g$. Therefore, the Joyce decomposition of $\g$ is $h$-orthogonal \cite{GP}. \par 
For the sake of computations, let us fix the orthogonal basis $\{v_1, e_2^1, e_3^1, e_4^1,\dots,v_m,\dots,e_4^m\}$ on $\ell \mathfrak{u}(1)\oplus\mathfrak{b}\oplus \bigoplus_{j=1}^m \mathfrak{d}_j$ as follows: $\{v_1, \dots, v_m\}$ is an orthonormal basis of $\ell \mathfrak{u}(1)\oplus\mathfrak{b}$, and $\{e_2^j, e_3^j, e_4^j\}$ is an orthogonal basis of $\mathfrak{d}_j$ for each $j$. Observe that the bi-invariance of $h$ forces $h(e_2^j, e_2^j)=h(e_3^j, e_3^j)=h(e_4^j,e_4^j)$.  
 Let $4d_j:=\dim_\R(\mathfrak{f}_j)$. On the subspaces $\mathfrak{f}_j$ we pick instead any orthonormal basis $\{f_1^j,\dots,f_{4d_j}^j\}$ adapted to $I,J,K$, that is, $f_{4k-2}^j=If_{4k-3}^j$, \ $f_{4k-1}^j=Jf_{4k-3}^j$, \ $f_{4k}^j=Kf_{4k-3}^j$, for any $k=1,\dots,d_j$. This is possible since the subspaces $\f_j$ are $h$-orthogonal and, by the bi-invariance of $h$ and the definition of $I,J,K$ on $\f_j$, the restriction $h_{| \f_j \times \f_j}$ is hyperhermitian. \par
It is straightforward to check that for all $j=1,\dots,m$ we have 
\[ [v_j, Jv_j]+\sum_{k=2}^4 [e_k^j,Je_k^j]= [e_2^j,Je_2^j]+[e_4^j,Je_4^j] =4e_3^j.\] 
Furthermore, it is computed in \cite{FG} that \[\sum_{k=1}^{4d_j} [f_k^j,Jf_k^j]=4d_je_3^j.\] We repeat the calculation for the readers' convenience. First of all we note that
\begin{align*}
&[f^j_{4k-3},Jf^j_{4k-3}]+[f^j_{4k-2},Jf^j_{4k-2}]+[f^j_{4k-1},Jf^j_{4k-1}]+[f^j_{4k},Jf^j_{4k}]\\
&\quad =[f^j_{4k-3},f^j_{4k-1}]-[f^j_{4k-2},f^j_{4k}]-[f^j_{4k-1},f^j_{4k-3}]+[f^j_{4k},f^j_{4k-2}]\\
&\quad =2\left([f^j_{4k-3},Jf^j_{4k-3}]+[f^j_{4k-2},Jf^j_{4k-2}]  \right).
\end{align*}
By definition of the hypercomplex structure on $ \mathfrak{f}_j $ and using the Jacobi identity, we have:
\[
\begin{split}
[f^j_{4k-3},Jf^j_{4k-3}]+[f^j_{4k-2},Jf^j_{4k-2}] &=[f^j_{4k-3},Kf^j_{4k-2}]-[f^j_{4k-2},Kf^j_{4k-3}]\\
&=[f^j_{4k-3},[e^j_4,f^j_{4k-2}]]-[f^j_{4k-2},[e^j_{4},f^j_{4k-3}]]\\
&=[e^j_4,[f^j_{4k-3},f^j_{4k-2}]]\,.
\end{split}
\]
Therefore, summing over $k$ we get:
\begin{equation}\label{eqn}
\sum_{i=1}^{4d_j} [f_i^j,Jf_i^j]=2\sum_{k=1}^{d_j} [e^j_4,[f^j_{4k-3},f^j_{4k-2}]]\,.
\end{equation}
To conclude, we show that for any fixed $ k\in\{1, \ldots, d_j\} $ the bracket $[e_4^j,[f^j_{4k-3},f^j_{4k-2}]]$ has non-zero component only along $e_3^j$. To do this we use the bi-invariant metric $h$ and the properties of the Joyce decomposition. First
\[
h([e^j_4,[f^j_{4k-3},f^j_{4k-2}]],X)=-h([f^j_{4k-3},f^j_{4k-2}],[e^j_4,X])\,,
\]
which clearly vanishes for all $X\in \ell\mathfrak{u}(1)\oplus \mathfrak{b}\oplus \bigoplus_{k\neq j} \mathfrak{d}_k\oplus \bigoplus_{l>j} \mathfrak{f}_l\oplus \langle e_4^j \rangle$ thanks to \eqref{eqn:joyce1}, \eqref{eqn:joyce2}, \eqref{eqn:joyce3}. On the other hand, if $X=e_2^j$ we have
\[
\begin{split}
h([e^j_4,[f^j_{4k-3},f^j_{4k-2}]],e_2^j)&=-2h([f^j_{4k-3},f^j_{4k-2}],e_3^j)\\
&=-2h(f^j_{4k-3},[f^j_{4k-2},e_3^j])\\
&=-2h(f^j_{4k-3},f^j_{4k})=0\,,
\end{split}
\]
whereas if $X=e_3^j$ we get
\[
\begin{split}
h([e^j_4,[f^j_{4k-3},f^j_{4k-2}]],e_3^j)&=2h([f^j_{4k-3},f^j_{4k-2}],e_2^j)\\
&=2h(f^j_{4k-3},[f^j_{4k-2},e_2^j])\\
&=2h(f^j_{4k-3},f^j_{4k-3})=2\,.
\end{split}
\]
Now, assume $X\in \mathfrak{f}_l$ with $ l< j $ then
\[
\begin{split}
h([e^j_4,[f^j_{4k-3},f^j_{4k-2}]],X)&=-h([e^j_4,[f^j_{4k-3},f^j_{4k-2}]],I^2X)\\
&=-h([e^j_4,[f^j_{4k-3},f^j_{4k-2}]],[e_2^l,IX])\\
&=\ \  h([e_2^l,[e^j_4,[f^j_{4k-3},f^j_{4k-2}]]],IX)\\
&=\ \  h([[e_2^l,e^j_4],[f^j_{4k-3},f^j_{4k-2}]],IX)+h([e^j_4,[e_2^l,[f^j_{4k-3},f^j_{4k-2}]]],IX)=0\,,
\end{split}
\]
where the last equality holds by using Jacobi identity and  \eqref{eqn:joyce2}--\eqref{eqn:joyce3}. Finally if $X\in \mathfrak{f}_j$
\[
\begin{split}
h([e^j_4,[f^j_{4k-3},f^j_{4k-2}]],X)&=-h([f^j_{4k-3},f^j_{4k-2}],[e^j_4,X])=-h([f^j_{4k-3},f^j_{4k-2}],KX)\\
&=-h([f^j_{4k-3},f^j_{4k-2}],[e^j_2,JX])=h([e^j_2,[f^j_{4k-3},f^j_{4k-2}]],JX)\\
&=\ \ h([[e^j_2,f^j_{4k-3}],f^j_{4k-2}],JX)+h([f^j_{4k-3},[e^j_2,f^j_{4k-2}]],JX)\\
&=\ \ h([f^j_{4k-2},f^j_{4k-2}],JX)-h([f^j_{4k-3},f^j_{4k-3}],JX)=0\,.
\end{split}
\]
Putting everything together we have shown that $[e^j_4,[f^j_{4k-3},f^j_{4k-2}]]=2 \frac{e^j_3}{\lambda_j^2} $, where $\lambda_j=\sqrt{h(e_3^j,e_3^j)}$. In order to compute $\eta$, we now consider the orthonormal basis $\{ e_1,\dots,e_{4n}\}=\{v_j, \frac{e_2^j}{\lambda_j},\dots,\frac{e_{4}^j}{\lambda_j},f_i^k\}$. We therefore obtain from \eqref{eqn}:
\[
\begin{split}
\sum_{i=1}^{4n} [e_i,Je_i]&=\sum_{j=1}^m \left([v_j, J v_j]+\sum_{k=2}^4 \frac{1}{\lambda_j^2} [e_k^j,Je_k^j]+\sum_{i=1}^{4d_j}[f_i^j,Jf_i^j] \right)\\
&=\sum_{j=1}^m \left(\frac{4}{\lambda_j^2} e_3^j +2  \sum_{k=1}^{d_j}[e_4^j,[f_{4k-3}^j,f_{4k-2}^j]] \right)\\
&=4\sum_{j=1}^m\frac{1}{\lambda_j^2}\left(1+d_j\right)e_3^j,
\end{split}
\]
implying
\[
\eta(X)=2\sum_{j=1}^{m}\frac{1}{\lambda_j^2}\left(1+d_j\right)h(JX,e_3^j)\,,
\]
as claimed.
\end{proof}

\begin{cor}\label{Cor:Leeclosed}
Let $G$ be a compact simple Lie group of rank $r$ and $(I,J,K)$ a Joyce hypercomplex structure on $\mathbb{T}^{\ell}\times G$. Then the Obata connection $1$-form is closed if and only if
\begin{equation}\label{eqn:SLconditions}
\sum_{j=1}^{m}\frac{1}{\lambda_j^2}\left(1+d_j\right)h(JX,e_3^j)=0, \qquad \text{for all } X\in \mathfrak{b}.
\end{equation}
In particular, when $(I,J,K)$ is a Joyce hypercomplex structure compatible with a bi-invariant metric $g$ constructed as in Section \ref{section:HKT}, $d\eta=0$ if and only if $\ell=m=r$, i.e. if and only if the abelian summand $\mathfrak{b}$ in the Joyce decomposition $\eqref{eqn:Joycedec}$ is trivial.
\end{cor}
\begin{proof} 
Since the Lie algebra $\mathfrak{g}$ is simple, the commutator of $\ell\mathfrak{u}(1)\oplus \g$ is $\g$, implying that 
\[
d\eta=-\eta\vert_{\mathfrak{g}}\,.
\]
On the other hand $\eta$ always vanishes on $\bigoplus_{j=1}^m \mathfrak{d}_j \oplus \bigoplus_{j=1}^m \mathfrak{f}_j$, thus $d\eta=0$ if and only if $\eta\vert_\mathfrak{b}=0$.

When $h$ coincides with a bi-invariant metric $g$ constructed as in Section \ref{section:HKT}, we can write
\[
\eta(X)=2\sum_{j=1}^{m}\frac{1}{\lambda_j^2}\left(1+d_j\right)g(JX,Je_1^j)=2\sum_{j=1}^{m}\frac{1}{\lambda_j^2}\left(1+d_j\right)g(X,e_1^j)
\]
which allows us to conclude that $d\eta=0$ if and only if $e_1^j\in \ell \mathfrak{u}(1)$ for all $j=1,\dots,m$, i.e. if and only if $\ell=2m-r=m$ \footnote{Recall that, in Section \ref{section:HKT}, the basis $\{e_1^j\}$ is chosen so that the first $\ell$ vectors are in $\ell\mathfrak{u}(1)$ and the remaining $m-r$ are in $\bi$.}. 
\end{proof}

Now we shall focus on those Joyce hypercomplex structures admitting a (strong) HKT metric obtained by extending the Killing--Cartan form (see Section~\ref{section:HKT}). \par

Such strong HKT metrics are related to the \emph{twisted Calabi--Yau system} introduced in \cite{GRST}. Indeed, on a hyperhermitian manifold $(M,I,J,K,g)$, the $(2,0)$ form $\Omega$ is non-degenerate, namely $\Psi=\Omega^n$ is a complex volume form of constant norm. Furthermore, the hyperhermitian structure is HKT if and only if $\partial \Omega=0$ \cite{GP} and in this case we have 
$$
d\Psi=\theta \wedge \Psi\,,
$$
which is the second equation of the twisted Calabi-Yau system \eqref{eqtCY} (in fact it is enough to have $\partial \Omega^{n-1}=0$ \cite{FG}). If the HKT metric is strong, then $\omega_I=g(I\cdot,\cdot)$ satisfies the pluriclosed condition $\partial \bar\partial \omega_I=0$. In conclusion, in order for the Killing--Cartan form to provide a solution to the system we have to further impose that $d\theta=0$.

We are ready to prove Theorem \ref{THM:B}:
\begin{thm}\label{main2}
Let $M$ belong to the list $\eqref{list}$. Then any left-invariant hypercomplex structure $(I,J,K)$ on $M$ has restricted Obata holonomy contained in ${\rm SL}(n,\H)$. Moreover, $(M,I)$ admits a left-invariant solution to the twisted Calabi-Yau system $\eqref{eqtCY}$. 
\end{thm}
\begin{proof}
Recall that any left-invariant hypercomplex structure on $M$ arises from the Joyce construction in Section \ref{sec:hcx}. 
According to \cite[Section 6]{SSTvP} (see also \cite[Section 5]{OP}) a compact simple Lie group has trivial $\bi$ in its Joyce decomposition if and only if it is one of the following:
\[
\mathrm{SO}(2k+1),
\ \  \mathrm{SO}(4k) , \ \ \mathrm{Sp}(k) ,\ \ \mathrm{E}_7, \ \  \mathrm{E}_8, \ \ \mathrm{F}_4, \ \ \mathrm{G}_2,
\]
and therefore any Joyce hypercomplex structure on a manifold $M$ in the list \eqref{list} satisfies the hypothesis of Corollary \ref{Cor:Leeclosed}. Indeed, the condition of having trivial summand $\mathfrak{b}$ in the Joyce decomposition has the consequence that every Joyce hypercomplex structure is compatible with an extension $g$ of the Killing--Cartan form (see Section \ref{section:HKT}). Since the Ricci tensor of the Obata connection is precisely $d\eta$, the Obata connection has restricted holonomy group in $\mathrm{SL}(n,\H)$.

Furthermore, under these hypotheses $g$ is strong HKT and $\eta$ coincides with the Lee form of $g$ \cite[Proposition 4.1]{IP}. 
We conclude that the pair $(g,\Psi=\Omega^n)$ provides a solution to the twisted Calabi-Yau equations \eqref{eqtCY}.
\end{proof}

\begin{rmk}
When $\bi \neq 0$ in the Joyce decomposition of $\g$, not every Joyce hypercomplex structure is constructed as in Section \ref{section:HKT}; however, there are infinitely many that are. For these, the restricted holonomy of the Obata connection $\mathrm{Hol}_0(\nabla)$ is never a subgroup of $\mathrm{SL}(n,\H)$, by Corollary \ref{Cor:Leeclosed}. This occurs in the following Joyce hypercomplex manifolds:
 \[
 S^1\times \mathrm{SU}(2k) \ (k \geq 2), \quad \mathrm{SU}(2k+1),   \quad \mathbb{T}^{2k-1}\times \mathrm{SO}(4k+2), \quad  \mathbb{T}^2 \times \mathrm{E}_6.
 \]
Except for $\mathrm{SU}(2k+1)$ all these manifolds admit Joyce hypercomplex structures with $\mathrm{Hol}_0(\nabla)\subseteq \mathrm{SL}(n,\H)$ and others with $\mathrm{Hol}_0(\nabla) \not \subseteq \mathrm{SL}(n,\H)$. The case of $\mathrm{SU}(2k+1)$ is different because there is no torus part. Being simply connected, $\mathrm{Hol}_0(\nabla)=\mathrm{Hol}(\nabla)$. Moreover, as shown in \cite{BDV,SV}, no Joyce hypercomplex manifold is an $\mathrm{SL}(n,\mathbb{H})$-manifold; in particular, no Joyce hypercomplex structure on $\mathrm{SU}(2k+1)$ has restricted holonomy contained in $\mathrm{SL}(n,\mathbb{H})$. This can also be seen from the criterion of Corollary \ref{Cor:Leeclosed}. 
\end{rmk}

\begin{ex}
Consider the Joyce decomposition
\[
\mathfrak{su}(4)=\mathfrak{b}\oplus \mathfrak{d}_1 \oplus \mathfrak{f}_1 \oplus \mathfrak{d}_2
\]
where $\dim(\mathfrak{b})=1$ and $\dim (\mathfrak{f}_1)=8$. Let $(I,J,K)$ be a Joyce hypercomplex structure on $S^1 \times \mathrm{SU}(4)$ associated to the basis $(e_1^1,e_1^2)$ of $\mathfrak{u}(1)\oplus \mathfrak{b}$ such that $X=ae_1^1+be_1^2$ is a generator of $\mathfrak{b}$. Let $h$ be any bi-invariant metric on $S^1 \times \mathrm{SU}(4)$.  Then, the condition \eqref{eqn:SLconditions} translates into the equation
\[
0=\frac{1}{\lambda_1^2}(1+8)h(ae_3^1+be_3^2,e_3^1)+\frac{1}{\lambda_2^2}h(ae_3^1+be_3^2,e_3^2)=9a+b
\]
thus, $\mathrm{Hol}_0(\nabla)\subseteq \mathrm{SL}(4,\H)$ if and only if $\mathfrak{b}=\langle e_1^1-9e_1^2 \rangle$.
\end{ex}

In order to produce more examples of solutions to the twisted Calabi-Yau equation, we can use the method introduced by Barberis and Fino in \cite{BF}.

\begin{cor}\label{CortCY}
Let $(\mathfrak{g},I,J,K,g)$ be the strong HKT Lie algebra of one of the groups listed in $\eqref{list}$, with $g$ bi-invariant, and let $\rho \colon \mathfrak{g} \to \mathfrak{sp}(r)$ be a Lie algebra homomorphism. Consider the semidirect product $\mathfrak{g}\ltimes_\rho \H^r$ and the hyperhermitian structure $(\tilde I,\tilde J,\tilde K,\tilde g)$, where $\tilde g$ is the product of $g$ and the natural inner product on $\H^r$ and for any $(X,q)\in \mathfrak{g}\ltimes_\rho \H^r$, let
\[
\tilde I(X,q)=(IX,iq)\,, \ \  \tilde J(X,q)=(JX,jq)\,,\ \  \tilde K(X,q)=(KX,kq)\, .
\]
Then $(\tilde I,\tilde J,\tilde K,\tilde g)$ is a strong HKT structure with closed Lee form. In particular it is a solution to the twisted Calabi-Yau equations. 
\end{cor}
\begin{proof}
It is clear that $(\tilde g,\tilde I,\tilde J,\tilde K)$ provides a hyperhermitian structure on $\mathfrak{g} \ltimes_\rho \H^r$. It has been proved in \cite[Theorem 3.3]{BF} that $\tilde g$ is strong HKT if and only if $g$ is. Finally, in \cite[Proposition 3.1]{BF} it has also been observed that the Lee form $\tilde \theta$ of $\tilde g$ coincides with $\theta \circ p$, where $\theta$ is the Lee form of $g$ and $p \colon \mathfrak{g}\ltimes_\rho \H^r \to \mathfrak{g} $ is the orthogonal projection. Thus we deduce that $d\tilde \theta=0$ if and only if $d\theta=0$.
\end{proof}

Barberis and Fino also discuss cases in which the Lie algebras resulting from Corollary \ref{CortCY} admit co-compact lattices, thus yielding more compact examples.
\begin{rmk}
As observed in \cite{ABB}, the Obata connections $\widetilde{\nabla}$ and $\nabla$ associated to the hypercomplex structures $(\widetilde{I}, \widetilde{J}, \widetilde{K})$ and $(I, J, K)$, respectively, have the same holonomy algebra:
\[
\mathfrak{hol}(\widetilde{\nabla}) \cong \mathfrak{hol}(\nabla).
\]
In particular, by Theorem \ref{main2}, it follows that
$
\mathfrak{hol}(\widetilde{\nabla}) \subseteq  \mathfrak{sl}(n,\mathbb{H}).
$
\end{rmk}

\bibliography{biblio}
\bibliographystyle{alpha}

\noindent
B. Brienza, L. Vezzoni \\
Dipartimento di Matematica ``G. Peano'', Universit\`{a} degli studi di Torino \\
Via Carlo Alberto 10, Torino (Italy) \\
Email: \texttt{beatrice.brienza@unito.it} \\
Email: \texttt{luigi.vezzoni@unito.it}

\vskip0.8em
\noindent
U. Fowdar \\ 
Institute of Mathematics, University of Warsaw, 
Banacha 2, 02-097 Warszawa (Poland)\\
Email: \texttt{u.fowdar@uw.edu.pl}

\vskip0.8em
\noindent
G. Gentili \\ 
Laboratoire de mathématiques d'Orsay, Université Paris-Saclay, 91405 Orsay (France)\\
Email: \texttt{giovanni.gentili@universite-paris-saclay.fr}
\end{document}